\documentclass[a4paper,reqno]{amsart}
\usepackage{amscd, amssymb, pb-diagram, amsmath}
\usepackage{amsthm}
\usepackage{slashed}
\usepackage[utf8]{inputenc}
\usepackage[pdftex]{graphicx}
\usepackage[bookmarks=true]{hyperref} 

\usepackage{tikz}
\usepackage[margin=3.7cm]{geometry}
\usepackage{color}
\usepackage{tensor}
\usetikzlibrary{arrows,matrix,backgrounds,decorations.markings}
\numberwithin{equation}{section}

\usepackage[all]{xy}
\usetikzlibrary{calc}

\DeclareUnicodeCharacter{00A0}{~}

\DeclareFontFamily{OT1}{pzc}{}
\DeclareFontShape{OT1}{pzc}{m}{it}{<-> s * [1.2] pzcmi7t}{}
\DeclareMathAlphabet{\mathpzc}{OT1}{pzc}{m}{it}

\newcommand{\<}{\langle}
\renewcommand{\>}{\rangle}

\def\id{\operatorname{id}}

\def\sup{\operatorname{sup}}

\def\id{\operatorname{id}}

\def\C{\mathbb{C}}

\def\R{\mathbb{R}}
\def\N{\mathbb{N}}
\def\Z{\mathbb{Z}}

\def\LL{\mathcal{L}}

\newcommand{\tn}{\textnormal}

\newcommand{\ot}{\otimes}
\newcommand{\otss}{\ot_B}

\newcommand{\cor}{\operatorname{Corr}}

\newcommand{\gae}{\lower 2pt \hbox{$\, \buildrel {\scriptstyle >}\over {\scriptstyle
\sim}\,$}}

\newcommand{\lae}{\lower 2pt \hbox{$\, \buildrel {\scriptstyle <}\over {\scriptstyle
\sim}\,$}}

\newcommand{\MU}[1]{
\setbox0\hbox{$#1$}
\setbox1\hbox{$W$}
\ifdim\wd0>\wd1 #1^{\sim} \else \widetilde{#1} \fi
}

\newcommand{\pre}[1]{{}_{#1}}

\newcommand{\Per}{{\bf P}}

\newcommand{\xn}{{\bf x}}
\newcommand{\yn}{{\bf y}}

\newcommand{\zn}{{\bf z}}

\newtheorem{thm*}{Theorem}
\newtheorem{thm}{Theorem}[section]

\newtheorem{corollary}[thm]{Corollary}

\newtheorem{lemma}[thm]{Lemma}
\newtheorem{prop}[thm]{Proposition}
\newtheorem{proposition}[thm]{Proposition}

\theoremstyle{definition}
\newtheorem{definition}[thm]{Definition}

\theoremstyle{remark}
\newtheorem{remark}[thm]{Remark}

\newtheorem{example}[thm]{Example}

\usepackage{adjustbox}
\usepackage{verbatim}
\usepackage{graphicx}
\usepackage{enumerate}
\usepackage{dcolumn,array}
\usepackage{mathtools, caption} 

\newcommand{\righttext}[1]{\qquad\text{#1 }}
\newcommand{\midtext}[1]{\qquad\text{#1 }\qquad}

\usepackage{relsize}
\usetikzlibrary{arrows}

\DeclarePairedDelimiterX{\norm}[1]{\lVert}{\rVert}{#1}

\begin{document}

\date{\today}
\title[Wieler solenoids]{Wieler solenoids: non-Hausdorff expansiveness, Cuntz-Pimsner models, and functorial properties}

\author{Robin J. Deeley}
\address{Robin J. Deeley,   Department of Mathematics,
University of Colorado Boulder
Campus Box 395,
Boulder, CO 80309-0395, USA }
\email{robin.deeley@colorado.edu}

\author{Menev\c se Ery\"uzl\"u}
\address{Menev\c se Ery\"uzl\"u,   Department of Mathematics,
University of Colorado Boulder
Campus Box 395,
Boulder, CO 80309-0395, USA }
\email{Menevse.Eryuzlu@colorado.edu}

\author{Magnus Goffeng}
\address{Magnus Goffeng,
Centre for Mathematical Sciences,
Lund University,
Box 118, 221 00 LUND, Sweden
}
\email{magnus.goffeng@math.lth.se}
\author{Allan Yashinski}
\address{Allan Yashinski,  Department of Mathematics, University of Maryland, College Park, MD 20742-4015, USA }
\email{ayashins@umd.edu}
\thanks{RJD was partially supported by NSF Grants DMS 2000057 and DMS 2247424. MG was supported by the Swedish Research Council Grant VR 2018-0350.}

\begin{abstract}
Building on work of Williams, Wieler proved that every irreducible Smale space with totally disconnected stable sets can be realized via a stationary inverse limit. Using this result, the first and fourth listed authors of the present paper showed that the stable $C^*$-algebra associated to such a Smale space can be obtained from a stationary inductive limit of a Fell algebra. Its spectrum is typically non-Hausdorff and admits a self-map related to the stationary inverse limit. With the goal of understanding the fine structure of the stable algebra and the stable Ruelle algebra, we study said self-map on the spectrum of the Fell algebra as a dynamical system in its own right. Our results can be summarized into the statement that this dynamical system is an expansive, surjective, local homeomorphism of a compact, locally Hausdorff space and from its $K$-theory we can compute $K$-theoretical invariants of the stable and unstable Ruelle algebra of a Smale space with totally disconnected stable sets. 
\end{abstract}

\maketitle

\section*{Introduction}
In \cite{Wil} Williams proved that every expanding attractor can be realized as a solenoid (i.e., a stationary inverse limit). The space in Williams' construction is a branched manifold and the map satisfies natural axioms. The expanding attractors studied by Williams are examples of Smale spaces that have totally disconnected stable sets. Building on Williams' axioms, Wieler \cite{Wie} proved that every Smale space with totally disconnected stable sets can be realized as a solenoid. Roughly speaking, Wieler replaces the differential geometric axioms of Williams with axioms that only rely on topological/metric space properties. This is necessary to include examples such as shifts of finite type where the solenoid is constructed from a Cantor set. In terms of notation, we let $(Y, g)$ denote the Wieler pre-solenoid and $(X, \varphi)$ denote the associated Smale space; as mentioned $(X, \varphi)$ is the stationary inverse limit of $(Y, g)$.

The stable $C^*$-algebra associated to a Smale space is introduced in \cite{Put} and \cite{PutSpi}. Building on work of Gon\c{c}alves \cite{Gon1, Gon2} (also see \cite{Min} and \cite{GonRamSol}) the first and fourth listed authors \cite{DeeYas} showed that the stable algebra is isomorphic to a stationary inductive limit when the Smale space is obtained from a Wieler pre-solenoid. The relevant $C^*$-algebra in the inductive limit is particularly nice; it is the Fell algebra associated to a local homeomorphism and furthermore the spectrum of this algebra is related to but typically not equal to the space in Wieler's construction. This paper aims at clarifying this relation.

More to the point, the spectrum of the Fell algebra is typically non-Hausdorff. However, it is compact and locally Hausdorff. In addition, the original map on the pre-solenoid lifts to a map on this compact, locally Hausdorff space. We denote this space and map as a pair $(X^u(\Per)/{\sim_0}, \tilde{g})$ where the notation is introduced in more detail in Section \ref{SecCstarWieler}. Importantly for us, the map $\tilde{g}$ is a local homeomorphism and (in a non-Hausdorff sense that is introduced in the present paper) expansive. In this sense, the dynamical system $(X^u(\Per)/{\sim_0}, \tilde{g})$ is a ``non-Hausdorff resolution'' of the failure of $(Y,g)$ to be an expansive local homeomorphism.

Thus, at the purely dynamical level, there is a trade-off between the Hausdorffness of the space of the pre-solenoid and the map being a local homeomorphism and expansive. For $C^*$-algebraic applications, the map being a local homeomorphism and expansive is often more important than the space being Hausdorff. Using $(X^u(\Per)/{\sim_0}, \tilde{g})$ as our primary example, we study the following three areas:\\

\emph{The notion of forward orbit expansiveness in the context of non-Hausdorff spaces.} This builds on the work of Achigar, Artigue, and Monteverde \cite{MR3501269} who study expansiveness for homeomorphisms on non-Hausdorff spaces.
We obtain natural conditions on the pre-solenoid (which is non-Hausdoff) that ensures that the associated solenoid is Hausdorff. In our prototypical example, we show that the solenoid associated with $(X^u(\Per)/{\sim_0}, \tilde{g})$ is the original Smale space. Thus any Smale space with totally disconnected stable sets can be realized as solenoid where the map of the pre-solenoid is an expansive, local homeomorphism; of course, the space cannot always be taken to be Hausdorff, see \cite[Example 10.3]{DeeYas}. This area is discussed in Section \ref{ExpDSNonHausSec}.\\

\emph{The inductive limit and Cuntz-Pimsner models for the stable and stable Ruelle algebra respectively.} Building on previous work in \cite{DGMW, DeeYas}, we relate stationary inductive limits to Cuntz-Pimsner algebras in Section \ref{cpsectionone}.
Our main result is a very general construction of a Cuntz-Pimsner model of a stationary inductive limit. Furthermore, we discuss the construction of unital Cuntz-Pimsner models under the assumption of the existence of a full projection. In our prototypical example $(X^u(\Per)/{\sim_0}, \tilde{g})$, the relevant Fell algebra always has a full projection as we show in Section \ref{fullprojectioninfell} and the stable algebra of the original Smale space is a stationary inductive limit of this Fell algebra and hence is the core of our Cuntz-Pimsner model of the stable Ruelle algebra.\\

\emph{The functorial properties of the $K$-theory of a compact, locally Hausdorff space.} The $K$-theory of a compact, locally Hausdorff space is defined using Fell algebras so functorial properties are subtle. In particular, a continuous map between compact, locally Hausdorff spaces need not induce a $*$-homomorphism at the $C^*$-algebra level. 
Nevertheless, we show that there are well-behaved right way and wrong way functors despite the failure of functoriality at the Fell algebra level. In our prototypical example, the map induced by $\tilde{g}$ plays a key role in computing the $K$-theory of the stable algebra and the stable Ruelle algebra of the original Smale space. These computations use the stationary inductive limit in the case of the stable algebra and the Cuntz-Pimsner models in the case of the stable Ruelle algebra; an explicit description of the wrong way map associated to $\tilde{g}$ is the key to such computations. This area is discussed in Section \ref{SecKthFun} where a number of illustrative examples are discussed at the end of the section.\\

For each of the three areas, we work at a quite general level. However, everything we consider is rooted and motivated by the dynamical system $(X^u(\Per)/{\sim_0}, \tilde{g})$ associated with a Wieler pre-solenoid. Before entering the above mentioned areas, we present some preliminary material in Section \ref{prelsection}, on Smale spaces, Fell algebras, groupoids and correspondences, as well as in Section \ref{SecCstarWieler}, on $C^*$-algebras associated with Wieler solenoids. The succeeding four sections of the paper treat the areas listed above and can be read independently of one another. Even if these last four sections of the paper are logically independent of each other, we believe they form a coherent picture of Wieler-Smale spaces. Such a description is missing in the literature except in the case of a Wieler-Smale space defined from a local homeomorphism \cite{DGMW}, see also Subsection \ref{SubSecWielerBasic} below. 

To avoid confusion about Hausdorffness, we indicate spaces that are possibly non-Hausdorff with a tilde. For instance, $X$ is Hausdorff while $\tilde{X}$ is possibly non-Hausdorff. The exception to this notation is $X^u(\Per)/{\sim_0}$ which in general can be non-Hausdorff.  \\

{\bf Acknowledgements}
The authors wish to thank Jamie Gabe for the examples leading up to Remark \ref{expaneeded}. The first listed author thanks the University of Hawaii and the Fields Institute for visits during which time the paper was completed.

\section{Preliminaries}
\label{prelsection}

In this section, we will present some preliminary material from the literature. As the paper aims for a rather broad view on Wieler-Smale spaces, we carefully review the known results. 

\subsection{Smale spaces} \label{Section-SmaleSpaces}
Although we are only interested in Wieler solenoids some definitions and basic properties of general Smale spaces are required. The reader can find more on Smale spaces in \cite{Kil, Put, PutSpi, Rue}.
\begin{definition} \label{SmaSpaDef}
A Smale space is a metric space $(X, d)$ along with a homeomorphism $\varphi: X\rightarrow X$ with the following additional structure: there exists global constants $\epsilon_X>0$ and $0< \lambda < 1$ and a continuous map, called the bracket map, 
\[
[ \ \cdot \  , \ \cdot \ ] :\{(x,y) \in X \times X : d(x,y) \leq \epsilon_X\}\to X
\]
such that the following axioms hold
\begin{itemize}
\item[B1] $\left[ x, x \right] = x$;
\item[B2] $\left[x,[y, z] \right] = [x, z]$ assuming both sides are defined;
\item[B3] $\left[[x, y], z \right] = [x,z]$ assuming both sides are defined;
\item[B4] $\varphi[x, y] = [ \varphi(x), \varphi(y)]$ assuming both sides are defined;
\item[C1] For $x,y \in X$ such that $[x,y]=y$, $d(\varphi(x),\varphi(y)) \leq \lambda d(x,y)$;
\item[C2] For $x,y \in X$ such that $[x,y]=x$, $d(\varphi^{-1}(x),\varphi^{-1}(y)) \leq \lambda d(x,y)$.
\end{itemize}
\end{definition}

A Smale space is denoted simply by $(X,\varphi)$ and to avoid certain trivialities, throughout the paper we assume that $X$ is infinite. 

\begin{definition}
Suppose $(X, \varphi)$ is a Smale space and $x$, $y$ are in $X$. We write 
\[ x \sim_s y \hbox{ if }\lim_{n \rightarrow \infty} d(\varphi^n(x), \varphi^n(y)) =0 \] 
and we write 
\[ x\sim_u y \hbox{ if } \lim_{n\rightarrow \infty}d(\varphi^{-n}(x) , \varphi^{-n}(y))=0. \] 
The $s$ and $u$ stand for stable and unstable  respectively.
\end{definition}

The global stable and unstable set of a point $x \in X$ are defined as follows:
\[
X^s(x) = \{ y \in X \: | \: y \sim_s x \} \hbox{ and } X^u(x)=\{ y \in X \: | \: y \sim_u x \}
\]
Furthermore, the local stable and unstable sets of $x$ are defined as follows: Given, $0< \epsilon \le \epsilon_X$, we have
\begin{align}
X^s(x, \epsilon) & = \{ y \in X \: | \: [x, y ]= y \hbox{ and }d(x,y)< \epsilon \} \hbox{ and } \\
X^u(x, \epsilon) & = \{ y \in X \: | \: [y, x]= y \hbox{ and } d(x,y)< \epsilon \}.
\end{align} 
The following result is a standard, see for example \cite{Put, Rue}.
\begin{thm} \label{wellKnownSmaleSpace} Suppose $(X, \varphi)$ is a Smale space and $x$, $y$ are in $X$ with $d(x,y)< \epsilon_X$. Then the following hold: for any $0< \epsilon \le \epsilon_X$
\begin{enumerate}
\item $X^s(x, \epsilon) \cap X^u(y, \epsilon)=\{[x, y]\}$ or is empty;
\item $\displaystyle X^s(x) = \bigcup_{n\in \N} \varphi^{-n}(X^s(\varphi^n(x), \epsilon))$;
\item $\displaystyle X^u(x) = \bigcup_{n\in \N} \varphi^n(X^u(\varphi^{-n}(x), \epsilon))$.
\end{enumerate}
\end{thm}
 A Smale space is mixing if for each pair of non-empty open sets $U$, $V$, there exists $N$ such that $\varphi^n(U)\cap V \neq \emptyset$ for all $n\geq N$. When $(X, \varphi)$ is mixing, $X^u(x)$ and $X^s(x)$ are each dense as subsets of $X$. However, one can use the previous theorem to give $X^u(x)$ and $X^s(x)$ locally compact, Hausdorff topologies, see for example \cite[Theorem 2.10]{Kil}.

\subsection{Background on Wieler solenoids} 
\label{SubSecWielerBasic}

Inspired by work of Williams \cite{Wil}, Wieler \cite{Wie} proved that every Smale space with totally disconnected stable sets $X^s(x)$ can be realized as a solenoid. Such Smale spaces will be referred to as Wieler solenoids or Wieler-Smale spaces. More precisely, she characterized Wieler-Smale spaces in terms of the following axioms on the pre-solenoid.

\begin{definition}[Wieler's Axioms] 
\label{WielerAxioms}
Let $(Y, \mathrm{d}_Y)$ be a compact metric space, and $g: Y \rightarrow Y$ be a continuous surjective map. Then, the triple $(Y, \mathrm{d}_Y, g)$ satisfies Wieler's axioms if there exists global constants $\beta>0$, $K\in \N^+$, and $0< \gamma < 1$ such that the following hold:
\begin{description}
\item[Axiom 1] If $x,y\in Y$ satisfy $\mathrm{d}_Y(x,y)\le \beta$, then 
\[
\mathrm{d}_Y(g^K(x),g^K(y))\leq\gamma^K \mathrm{d}_Y(g^{2K}(x),g^{2K}(y)).
\]
\item[Axiom 2] For all $x\in V$ and $0<\epsilon\le \beta$
\[
g^K(B(g^K(x),\epsilon))\subseteq g^{2K}(B(x,\gamma\epsilon)).
\] 
\end{description}
\end{definition}

\begin{definition} 
\label{WieSolenoid}
Suppose $(Y,\mathrm{d}_Y, g)$ satisfies Wieler's axioms and form the inverse limit space 
\[
X:= \varprojlim (Y, g) = \{ (y_n)_{n\in \N} = (y_0, y_1, y_2, \ldots ) \: | \: g(y_{i+1})=y_i \hbox{ for each }i\ge0 \}.
\]
Consider the map $\varphi: X \rightarrow X$ defined via
\[
\varphi(x_0, x_1, x_2, \ldots ) = (g(x_0), g(x_1), g(x_2), \ldots) = (g(x_0), x_0, x_1, \ldots ).
\] 
Following Wieler, we take a metric on $X$, $\mathrm{d}_X$, given by
\[
\mathrm{d}_X((x_n)_{n\in \N}, (y_n)_{n\in \N} ) = \sum_{i=0}^K \gamma^{i} \mathrm{d}^{\prime}_X( \varphi^{i}(x_n)_{n\in \N}, \varphi^{i}(y_n)_{n\in \N}),
\]
where $\mathrm{d}^{\prime}_X ( (x_n)_{n\in \N}, (y_n)_{n\in \N} )= \sup_{n\in \N} (\gamma^n \mathrm{d}_Y(x_n, y_n))$. We note that the topology induced by $\mathrm{d}_X$ is the product topology. The triple $(X, \mathrm{d}_X, \varphi)$ is called a Wieler solenoid. 
\end{definition}

\begin{remark}
We will assume that $Y$ is infinite. In particular, this ensures that $X$ is infinite and that $g$ is not a homeomorphism, see \cite{DeeYas}. The pair $(Y, g)$ will be called a presolenoid and $(X,\varphi)$ the associated solenoid or Smale space.
\end{remark}

\begin{thm}\cite[Theorems A and B on page 4]{Wie} 
\label{WielerTheorem}
Suppose that $(Y, \mathrm{d}_Y, g)$ satisfies Wieler's axioms. Then the associated Wieler solenoid $(X, \mathrm{d}_X, \varphi)$ is a Smale space with totally disconnected stable sets. The constants in Wieler's definition give Smale space constants: $\epsilon_X=\frac{\beta}{2}$ and $\lambda=\gamma$. Moreover, if $\xn=(x_n)_{n\in\N} \in X$ and $0< \epsilon \le \frac{\beta}{2}$, the locally stable and unstable sets of $(X, \mathrm{d}_X, \varphi)$ are given by
\[
X^s( \xn , \epsilon)= \{ \yn=(y_n)_{n\in \N} \: | \: y_m= x_m \hbox{ for }0\le m \le K-1 \hbox{ and } \mathrm{d}_X(\xn,\yn) \le \epsilon \}
\]
and
\[
X^u(\xn, \epsilon) =\{  \yn=(y_n)_{n\in \N} \: | \: \mathrm{d}_Y(x_n,y_n)< \epsilon \: \forall n \hbox{ and } \mathrm{d}_X(\xn,\yn) \le \epsilon \}
\]
respectively. 

Conversely, if $(X, \varphi)$ is an irreducible Smale space with totally disconnected stable sets, then there exists a triple $(Y, \mathrm{d}_Y, g)$ satisfying Wieler's axioms such that $(X, \varphi)$ is conjugate to the Wieler solenoid associated to $(Y, \mathrm{d}_Y, g)$.
\end{thm}

\begin{remark}
Wieler's axioms and the previous theorem should be compared with work of Williams \cite{Wil}. As mentioned in the introduction, an important difference between the two is that Wieler's are purely metric space theoretic. If a triple $(Y, \mathrm{d}_Y, g)$ satisfies Williams' axioms the inverse limit space $X$ with $\varphi$ as in definition \ref{WieSolenoid} is also Smale space and we will refer to such Smale spaces as Williams solenoids. However, we are most concerned with the more general case of a Wieler solenoid so we will not review Williams' axioms but rather direct the reader to \cite{Wil} for more details.
\end{remark}

An important special case occurs when $g$ is a local homeomorphism that satisfies Wieler's axioms. This special case was studied in detail in \cite{DGMW} (also see \cite[Section 4.5]{ThoAMS}). We recall a salient characterization of how refinements of Wieler's axioms (Definition \ref{WielerAxioms}) are equivalent to $g$ being a local homeomorphism, more detailed statements can be found in \cite[Section 3]{DGMW}.

\begin{thm}[Lemma 3.7 and 3.8 of \cite{DGMW}]
\label{equiforlocalhom}
Let $(Y, \mathrm{d}_Y)$ be a compact metric space, and $g: Y \rightarrow Y$ be a continuous surjective map. The following are equivalent:
\begin{itemize}
\item $(Y,\mathrm{d}_Y, g)$ satisfies Wieler's axioms and $g$ is a local homeomorphism.
\item $(Y,\mathrm{d}_Y, g)$ satisfies Wieler's axiom 1 and $g$ is open.
\item $(Y,\mathrm{d}_Y, g)$ satisfies Wieler's axiom $2$ and $g^K$ is locally expanding (for the $K$ in Wieler's axiom 2).
\end{itemize}
\end{thm}

\begin{remark}
The reader is encouraged to compare Theorem \ref{equiforlocalhom} to the more satisfying situation arising from going to a non-Hausdorff setting in Theorem \ref{WielerNonHausOrbExp}.
\end{remark}

A list of examples of Wieler solenoids can be found in \cite{DeeYas}. Three explicit examples that are relevant and illustrative of the results in this paper are the following:

\begin{example}[$n$-solenoid] \label{nSolEx}
Let $S^1 \subseteq \C$ be the unit circle.  Take $n > 1$ and define 
$g: S^1 \rightarrow S^1$ via $z \mapsto z^n$.
Since $g$ is open and expansive (notice that $|g'(z)|=n$), one readily verifies Wieler's axioms for $(S^1, g)$ (see Theorem \ref{equiforlocalhom}). Hence the associated inverse limit is a Smale space. It is worth emphasizing that in this case $g$ is a local homeomorphism.
\end{example}

\begin{example}[$ab/ab$-solenoid] 
\label{ababSolEx}
Let $Y = S^1 \vee S^1$ be the wedge sum of two circles as in Figure \ref{Figure-ab/ab-PreSolenoid}.

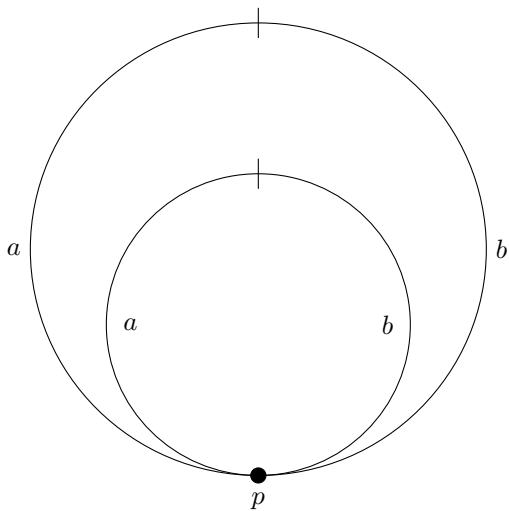
\begin{figure}[h]
\begin{tikzpicture}
\draw (0,0) circle [radius=3];
\draw (0,-1) circle [radius=2];
\draw[fill] (0,-3) circle [radius=0.1];
\draw (0,0.8) -- (0,1.2);
\draw (0,2.8) -- (0,3.2);
\node [below] at (0,-3.1) {$p$};
\node [left] at (-3.0,0) {$a$};
\node [right] at (3.0,0) {$b$};
\node [right] at (-1.9,-1) {$a$};
\node [left] at (1.9,-1) {$b$};
\end{tikzpicture}
\caption{$ab/ab$ pre-solenoid}
\label{Figure-ab/ab-PreSolenoid}
\end{figure}

The map $g: Y \rightarrow Y$ is defined using Figure \ref{Figure-ab/ab-PreSolenoid}. In Figure \ref{Figure-ab/ab-PreSolenoid}, we consider the outer circle to be the $a$-circle and the inner circle to be the $b$-circle. Each line segment labelled with $a$ in Figure \ref{Figure-ab/ab-PreSolenoid} is mapped onto the $a$-circle (i.e., the outer circle); while, each line segment labelled with $b$ in Figure \ref{Figure-ab/ab-PreSolenoid} is mapped onto the $b$-circle (i.e., the inner circle). The mapping is done in an orientation-preserving way, provided we have oriented both circles the same way, say clockwise. Note that $g$ is not a local homeomorphism in this example. For more details on this specific example and one-solenoids in general, see \cite{ThoSol, WilOneDim, Yi}. The next example is also of this form.
\end{example}

\begin{example}[$aab/ab$-solenoid] \label{aababSolEx}
Again, we take $Y = S^1 \vee S^1$ to be the wedge sum of two circles but with labels as in Figure \ref{Figure-aab/ab-PreSolenoid}. This example has been studied in \cite{DeeGofYasFellAlgPaper,DeeYas}.

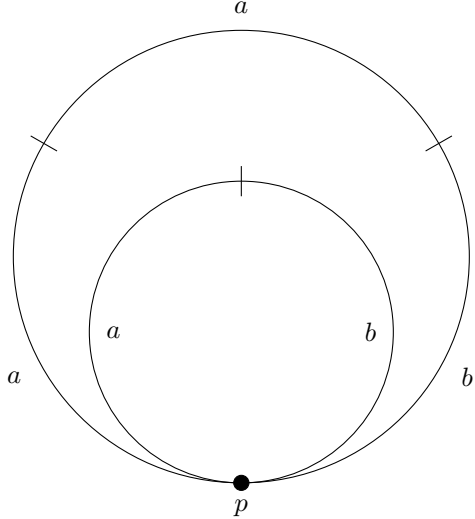
\begin{figure}[h]
\begin{tikzpicture}
\draw (0,0) circle [radius=3];
\draw (0,-1) circle [radius=2];
\draw[fill] (0,-3) circle [radius=0.1];
\draw (0,.8) -- (0,1.2);
\draw (-2.42487,1.4) -- (-2.77128,1.6);
\draw (2.42487,1.4) -- (2.77128,1.6);
\node [below] at (0,-3.1) {$p$};
\node [above] at (0,3.1) {$a$};
\node [left] at (-2.77128,-1.6) {$a$};
\node [right] at (2.77128,-1.6) {$b$};
\node [right] at (-1.9,-1) {$a$};
\node [left] at (1.9,-1) {$b$};
\end{tikzpicture}
\caption{$aab/ab$ pre-solenoid}
\label{Figure-aab/ab-PreSolenoid}
\end{figure}

The map $g: Y \rightarrow Y$ is defined from Figure \ref{Figure-aab/ab-PreSolenoid} via the same process as in Example \ref{ababSolEx}. The map $g$ is of course different than the one in Example \ref{ababSolEx} because the labels are different. Again, the resulting map $g$ is not a local homeomorphism and is an example of a one-solenoid, again see \cite{ThoSol, WilOneDim, Yi} for more on this class of examples.
\end{example}

\subsection{Fell algebras and their spectrum}
\label{fellalgsubsec}

Let us recall the basic facts about Fell algebras that we make use of. The spectrum of a separable $C^*$-algebra $A$ is defined as the set $\hat{A}$ of equivalence classes of irreducible representations (see \cite[Chapter 3]{Dix:C*}). The spectrum $\hat{A}$ can be topologized in several different ways, for instance in the Fell topology or the Jacobson topology. For a Fell algebra, the topologies coincide. The spectrum $\hat{A}$ is locally quasi-compact by  \cite[Corollary 3.3.8]{Dix:C*}.

A $C^*$-algebra $A$ is a Fell algebra if every $[\pi_0]\in \hat{A}$ admits a neighborhood $U$ and an element $b\in A$ such that $\pi(b)$ is a rank one projection for all $[\pi]\in U$. An equivalent definition of a Fell algebra is that $A$ is generated by its abelian elements. For details, see \cite[Chapter 3]{HKS}. A Fell algebra has locally Hausdorff spectrum (i.e. any $[\pi]\in \hat{A}$ has a Hausdorff neighborhood) by \cite[Corollary 3.4]{archsom}. The spectrum of a $C^*$-algebra is always locally quasi-compact. The properties of the spectrum of a Fell algebra can be summarized as being locally Hausdorff and locally locally compact (see \cite[Chapter 3]{CHR}). 

\begin{definition}
\label{hausres}
Let $\tilde{Y}$ be a topological space. A Hausdorff resolution of $\tilde{Y}$ is a surjective local homeomorphism $\psi:X\to \tilde{Y}$ from a locally compact, Hausdorff space $X$.
\end{definition}

\begin{example}
\label{fellex}
The main example of a Fell algebra that we will concern ourselves with arises in a rather explicit way from a Hausdorff resolution. The construction can be found in \cite[Corollary 5.4]{CHR}. Suppose that $\psi:X\to \tilde{Y}$ is a Hausdorff resolution of a topological space $\tilde{Y}$. It follows that $\tilde{Y}$ is locally Hausdorff and locally locally compact, and second countable if $X$ is. We define the equivalence groupoid
$$R(\psi):=X\times_\psi X:=\{(y_1,y_2)\in X\times X: \psi(y_1)=\psi(y_2)\}.$$
By declaring the domain and range mappings $d(y_1,y_2):=y_2$ and $r(y_1,y_2):=y_1$ to be local homeomorphisms, $R(\psi)$ becomes an etale groupoid over $X$. By \cite[Corollary 5.4]{CHR}, $C^*(R(\psi))$ is a Fell algebra with vanishing Dixmier-Douady invariant and spectrum $\tilde{Y}$. We also note that $R(\psi)$ is amenable so $C^*(R(\psi))=C^*_r(R(\psi))$.
\end{example}

The theory of Dixmier-Douady invariants of Fell algebras was introduced and developed in \cite{HKS}, also see \cite{CHR}. We only need to consider Fell algebras with vanishing Dixmier-Douady class in which case the following theorem (see \cite{HKS, CHR}) reduces the problem to a more manageable situation.

\begin{thm} 
\label{fellchar}
We have the following relationship between Fell algebras and non-Hausdorff spaces.
\begin{enumerate}
\item[a.] Let $A$ be a separable Fell algebra with vanishing Dixmier-Douady invariant. Then the locally Hausdorff and locally locally compact space $\hat{A}$ determines $A$ up to stable isomorphism in the sense that whenever $A'$ is a separable Fell algebra with vanishing Dixmier-Douady invariant then a homeomorphism $h:\hat{A}\to \widehat{A'}$ can be lifted to a stable isomorphism $A\otimes \mathbb{K}\cong A'\otimes \mathbb{K}$.
\item[b.] A topological space $\tilde{Y}$ is locally Hausdorff and locally locally compact if and only if it admits a Hausdorff resolution $\psi:X\to \tilde{Y}$. If $\tilde{Y}$ is second countable then $X$ can also be choosen second countable. In particular, any second countable, locally Hausdorff, locally locally compact topological space is the spectrum of a separable Fell algebra with vanishing Dixmier-Douady invariant.
\end{enumerate}
\end{thm}

\begin{remark}
\label{remarkonfellchar}
One can view the previous theorem as follows. Taking the spectra defines an equivalence between the category of stable isomorphism classes of separable Fell algebras with vanishing Dixmier-Douady invariants and that of second countable, locally Hausdorff and locally locally compact spaces. However, it should be emphasized that the morphisms in the category in the previous sentence are homeomorphisms; one cannot generalize to the case of continuous map, even in the case of locally Hausdorff and compact spaces. 
\end{remark}

For a proof of the first statement on stable uniqueness, see \cite[Theorem 7.13]{HKS}. The existence result in the second statement can be found in \cite[Corollary 5.5]{CHR}. 

\subsection{Correspondences}
Results from this section and the next will not be needed until Section \ref{cpsectionone}. Furthermore, the reader only interested in the purely dynamical results of the present paper can skip to Section \ref{SecCstarWieler} without issue.

A $C^*$-correspondence $\pre AE_B$ is a right Hilbert $B$-module equipped with a left action given by a homomorphism $\varphi_E: A\rightarrow \LL(E)$, where $\LL(E)$ denotes the $C^*$-algebra of adjointable operators on $E$. In the literature one also finds the term $A-B$-Hilbert $C^*$-module for a $C^*$-correspondence from $A$ to $B$. A $C^*$-correspondence homomorphism $\pre AE_B  \rightarrow \pre AF_B$ is a $B$-linear map $\Phi: E\rightarrow F$  satisfying
 \[ \Phi(a\cdot \xi)=a\cdot \Phi(\xi) \hbox{ and } \<\xi, \nu\>_C  = \<\Phi(\xi), \Phi(\nu)\>_C, \]
for all $a\in A$, and $\xi, \nu\in E.$

\begin{definition}\label{precor} An $A-B$ bimodule $E_0$ is called a \emph{pre-correspondence} if it has a $B$-valued semi-inner product satisfying 
\[ \<\xi, \nu\cdot b\> = \<\xi,\nu\> b ,   \hspace{.5cm} \<\xi,\nu\>^*=\<\nu ,\xi\> \]
and $\<a\cdot \xi, a\cdot \xi\> \leq \norm{a}^2\<\xi,\xi\>$ for all $a\in A, b\in B$ and $\xi,\nu \in E_0$. Modding out by the elements of length $0$ and completing gives a $C^*$-correspondence $\pre AE_B$ . We call $\pre AE_B$\ the \emph{completion} of the pre-correspondence $E_0$. 
\end{definition}

\begin{proposition}\label{preProp}\tn{\cite[Lemma 1.23]{enchilada}}
Let $E_0$ be an $A-B$ pre-correspondence given with the completion $\pre AE_B$ , and  let  $F$ be an $A-B$ correspondence. If there is a map $\Phi : E_0 \rightarrow F$ satisfying 
\[ \Phi (a\cdot \xi) = \varphi_Z (a) \Phi(\xi)  \midtext{and} \<\Phi(\xi), \Phi(\nu) \>_B = \<\xi, \nu\>_{B} , \]
for all $a\in A$ and $\xi,\nu\in E_0$, then $\Phi$ extends uniquely to an injective $A-B$ correspondence homomorphism $\tilde{\Phi}: E \rightarrow F. $
\end{proposition}

The \emph{balanced tensor product} $E\otimes_BF$ of
an $A-B$ correspondence $E$ and a $B-C$ correspondence $F$ is
formed as follows:
the algebraic tensor product $E\odot F$
is a pre-correspondence with the $A-C$ bimodule structure satisfying
\[
a(\xi\otimes \nu)c=a\xi\otimes \nu c
\righttext{for}a\in A,\xi\in E,\nu\in F,c\in C,
\]
and the unique $C$-valued semi-inner product whose values on elementary tensors are given by
\[
\<\xi_1\otimes \nu_1 ,\xi_2 \otimes \nu_2\>_C=\<\nu_1,\<\xi_1,\xi_2\>_B\cdot \nu_2\>_C
\righttext{for}\xi_1, \xi_2 \in E,\nu_1,\nu_2\in F.
\]

This semi-inner product defines a $C$-valued inner product on the quotient $E{\odot}_BF$ of $E\odot F$ by the subspace generated by elements of form 
\[ \xi\cdot b \otimes \nu - \xi\otimes \varphi_Y(b)\nu \righttext{($\xi\in E$, $\nu\in F$, $b\in B$)}.\]
The completion $E\otimes_B F$ of  $E{\odot}_BF$ with respect to the norm coming from the $C$-valued inner product is an $A-B$ correspondence, where the left action is given by 
 \[A\rightarrow \LL(E\otss F),  \righttext{$a\mapsto \varphi_E(a)\ot 1_F,$}\]
for $a\in A.$ In other words, the $A-C$ correspondence $E\otss F$ is the completion of the pre-correspondence $E\odot F$ (as in Definition \ref{precor}). 

We denote the canonical image of $\xi\otimes \nu$ in $E\otss F$ by $\xi\otss \nu$. The term \emph{balanced} refers to the property
\[
\xi\cdot b\otss \nu=\xi\otss b\cdot \nu
\righttext{for}\xi\in E,b\in B,\nu\in F,
\]
which is a consequence of the construction.

\begin{definition}\cite[Definition 5.7]{Sk}\label{mor} 
Hilbert modules $E_A$ and $F_B$ are Morita equivalent if there exists an imprimitivity bimodule $\pre BM_C$ such that $E\otss M \cong F$ as Hilbert C-modules. \end{definition}

We now put Definition~\ref{mor} in the setting of $C^*$-correspondences:

\begin{definition}
$C^*$-correspondences $\pre AE_B$\ and $\pre AF_C$ are Morita equivalent if there exists an imprimitivity bimodule $\pre BM_C$ such that $E\otss M \cong F$ as $A-C$ correspondences. 
\end{definition}

\subsection{Groupoid Actions and Equivalence }

The correspondences reviewed in the last section will in this paper arise from groupoids and their actions. Again, the reader only interested in the purely dynamical results can skip to Section \ref{SecCstarWieler} without issue.

\begin{definition}Suppose $G$ is a groupoid and that $X$ is a set together with a map $\rho: X\rightarrow G^{(0)}$ called the \emph{moment map}. Then a left action of $G$ on $X$ is a map $(g,x)\mapsto g\cdot x$ from $G*X=\{ (g,x)\in G\times X: s(g)=\rho(x)\}$ to $X$ such that 
\begin{itemize}
\item $\rho(x)\cdot x = x$, for all $x\in X$ and 
\item if $(g,g')\in G^{(2)}$ and $(g',x)\in G*X$, then $(g, g'\cdot x)\in G*X$ and $gg'\cdot x= g\cdot(g'\cdot x). $
\end{itemize}
\emph{Right actions are defined analogously, and we denote by $\sigma$ the moment map for a right action.} 
\end{definition}

\begin{definition}\cite[Definition 1.2]{M} 
\label{groupoidcorr}
Let $G_1$ and $G_2$ be second countable locally compact Hausdorff groupoids and $Z$ a second countable locally compact Hausdorff space. The space $Z$ is a groupoid correspondence from $G_1$ to $G_2$ if it satisfies the following properties:
\begin{enumerate}
\item there exists a left proper action of $G_1$ on $Z$ such that $\rho$ is an open map;
\item there exists a right proper action  of $G_2$ on Z;
\item the $G_1$ and $G_2$ actions commute;
\item the map $\rho$ induces a bijection of $Z/G_2$ onto $G_1^{(0)}$. 
\end{enumerate}
\end{definition}

\begin{thm}\tn{\cite[Theorem 1.4]{M}}\label{corGroupoid} Let $G_1$, $G_2$ be second countable locally compact Hausdorff etale groupoids; and let $Z$ be a groupoid correspondence from $G_1$ to $G_2$. Then the pre-correspondence $\pre {C_c(G_2)}C_c(Z)_{C_c(G_1)}$ extends to a correspondence from $C^*(G_2)$ to $C^*(G_1)$ with the actions
\begin{align*}
(\xi\cdot a)(z) &= \sum_{\substack{g\in G_1 with \\s(g)=\rho(z)}} \xi(g\cdot z)a(g)\\
(b\cdot \xi)(z)&=\sum_{\substack{g\in G_2 with \\ \sigma(z)=r(g^{-1})}} b(g^{-1})\xi(z\cdot g^{-1})
\end{align*}
for $\xi\in C_c(Z), a\in C_c(G_1), b\in C_c(G_2), z\in Z.$ The inner product is defined by  
\[\<\xi_1, \xi_2\>(g) = \sum_{\substack{h^{-1}\in G_2 with \\  r(h^{-1})=\sigma(z)}} \overline{\xi_1(z\cdot h^{-1})}\xi_2 (g^{-1}\cdot z\cdot h^{-1}),\]
where $g\in G_1, \xi_1, \xi_2\in C_c(Z),$ and  $z\in Z$ such that $r(g)=\rho(z)$.
\end{thm}

\section{$C^*$-algebras associated to a Wieler solenoid}
\label{SecCstarWieler}

The fine structure of the stable algebra of a Wieler solenoid was studied in \cite{DeeYas}, extending results from \cite{DGMW} in the case of the pre-solenoid being defined from a local homeomorphism. We here review and refine the relevant points of \cite{DeeYas}, leading up to the stable algebra of a Wieler solenoid being a stationary inductive limit of a Fell algebra defined from the dynamics. This Fell algebra plays an important role in the paper, as its spectrum will be the non-Hausdorff dynamical system $(X^u(\Per))/{\sim_0},\tilde{g})$ of main interest in this paper. 

\subsection{The stable algebra of a Smale space}
Following \cite{PutSpi}, we construct the stable groupoid of $(X, \varphi)$. Let $\Per$ denote a finite $\varphi$-invariant set of periodic points of $\varphi$ and define
\[
X^u(\Per)=\{ x \in X \: | \: x \sim_u p \hbox{ for some }p \in \Per \}
\]
and
\[
G^s(\Per) := \{ (x, y) \in X^u(\Per) \times X^u(\Per) \: | \: x \sim_s y \}.
\] 
Still following \cite{PutSpi}, a topology is defined on $G^s(\Per)$ by constructing a neighborhood base. Suppose $(x,y)\in G^s(\Per)$. Then there exists $k\in \N$ such that 
\[
\varphi^k(x) \in X^s\left(\varphi^k(y), \frac{\epsilon_X}{2}\right).
\]
Since $\varphi$ is continuous there exists $\delta>0$ such that 
\[
\varphi^k( X^u(y, \delta)) \subseteq X^u\left(\varphi^k(y), \frac{\epsilon_X}{2}\right).
\]
Using this data, we define a function $h_{(x,y,\delta)} : X^u(y, \delta) \rightarrow X^u(x, \epsilon_X)$ via
\[ 
z \mapsto \varphi^{-k} ( [ \varphi^k(z) , \varphi^k(x) ]) 
\]
and have the following result from \cite{PutSpi}:
\begin{thm} \label{etaleTopThm}
The function $h=h_{(x,y,\delta)}$ is a homeomorphism onto its image and (by letting $x$, $y$, and $\delta$ vary) the sets
\[
V(x,y,h, \delta) := \{ ( h(z), z ) \: | \: z \in X^u(y, \delta) \} 
\]
forms a neighborhood base for an \'{e}tale topology on the groupoid $G^s(\Per)$. Moreover, the groupoid $G^s(\Per)$ is amenable, second countable, locally compact, and Hausdorff.
\end{thm}

\begin{definition}
The stable Ruelle groupoid is the groupoid $G^s(\Per) \rtimes \Z$ where the $\Z$-actions is the one induced from $\varphi|_{X^u(P)}$; the associated $C^*$-algebra is call the stable Ruelle algebra. It is worth noting that this definition requires that $\Per$ is $\varphi$-invariant (as was assumed above).
\end{definition}

\subsection{The open subrelation, Fell algebra and its spectrum} \label{tildeZeroSection}

We review the construction of $\sim_0$ and the associated groupoid $C^*$-algebra studied in \cite{DeeYas}.

Let $\epsilon_Y>0$ be the global constant defined in \cite{DeeYas} and $\pi_0 : X^u(\Per) \rightarrow Y$ denote the map defined via
\[ 
\xn = (x_n)_{n\in \N} \mapsto x_0.
\]
\begin{definition} \label{tildeZeroDef}
Suppose $\xn$ and $\yn$ are in $X^u(\Per)$. Then $\xn \sim_0 \yn$ if 
\begin{enumerate}
\item $\pi_0(\xn)=\pi_0(\yn)$ (i.e., $x_0 = y_0$);
\item there exists $0< \delta_{\xn} < \epsilon_Y$ and open set $U \subseteq X^u(\yn, \epsilon_Y)$ such that
\[ 
\pi_0 ( X^u(\xn, \delta_{\xn})) =\pi_0 (U). 
\]
\end{enumerate}
Let $G_0(\Per) = \{ (\xn, \yn) \: | \: \xn \sim_0 \yn \}$.
\end{definition}

To parse the requirements of Definition \ref{tildeZeroDef}, the reader can return to Theorem \ref{WielerTheorem} for a description of the local unstable sets. Results in \cite{DeeYas} imply that $G_0(\Per)$ is an open subgroupoid of $G^s(\Per)$ and hence that $C^*(G_0(\Per))$ is a subalgebra of $C^*(G^s(\Per))$. Furthermore, we can define 
\[ G_k(\Per) = \{ (\xn, \yn) \: | \: \varphi^k(\xn) \sim_0 \varphi^k(\yn) \} \] 
and one of the main results of \cite{DeeYas} is the following:

\begin{thm} 
\label{MainDeeYas}
Using the notation above, there is a nested sequence of \'{e}tale subgroupoids
\[ G_0(\Per) \subset G_1(\Per) \subset G_2(\Per) \subset \ldots \]
of $G^s(\Per)$ such that $G^s(\Per) = \bigcup_{k=0}^\infty G_k(\Per)$ and each $G_k(\Per)$ is isomorphic to $G_0(\Per)$ in the natural way $(\xn, \yn) \mapsto (\varphi^{k}(\xn),\varphi^{k}(\yn))$.
\end{thm}

In practical terms, this theorem reduces the study of $G^s(\Per)$ to $G_0(\Per)$ and likewise the study of the $C^*$-algebra $C^*(G^s(\Per))$ to $C^*(G_0(\Per))$. Since $C^*(G_0(\Per))$ is a type I $C^*$-algebra many of its properties can be determined from its spectrum, $X^u(\Per))/{\sim_0}$. 

Furthermore, again following \cite{DeeYas}, we define $\tilde{g} : X^u(\Per))/{\sim_0} \rightarrow X^u(\Per))/{\sim_0}$ via
\[
[\xn] \mapsto [(g(x_0), g(x_1), \ldots )]
\]
and $r: X^u(\Per))/{\sim_0} \rightarrow Y$ via
\[
[\xn] \mapsto x_0.
\]
Proofs that $\tilde{g}$ and $r$ are well-defined can be found in \cite{DeeYas}. Moreover, there is a commutative diagram
\[ \begin{CD}
X^u(\Per) @>\varphi >> X^u(\Per) \\
@Vq VV @Vq VV \\
X^u(\Per)/{\sim_0} @>\tilde{g} >> X^u(\Per)/{\sim_0} \\
@Vr VV @Vr VV \\
Y @>g>> Y
\end{CD}\]

\begin{prop}
The maps $q: X^u(\Per) \rightarrow X^u(\Per)/{\sim_0}$ and $\tilde{g} : X^u(\Per))/{\sim_0} \rightarrow X^u(\Per))/{\sim_0}$ are each a local homeomorphisms (but not in general covering maps). In particular, $q: X^u(\Per) \rightarrow X^u(\Per)/{\sim_0}$ is a Hausdorff resolution of $X^u(\Per)/{\sim_0}$.
\end{prop}

Next, we discuss $G_0(\Per)$ and $X^u(\Per)/{\sim_0}$ for the three examples considered in Section \ref{SubSecWielerBasic}.

\begin{example}
Recall the setup of Example \ref{nSolEx}. The space $Y$ is the unit circle, $S^1 \subseteq \C$ and (with fixed $n > 1$)  
$g: S^1 \rightarrow S^1$ is defined via $z \mapsto z^n$. In this case, we take $\Per$ to be the set containing the single point $(1,1,1, \ldots)$. Then $X^u(\Per)$ is homeomorphic to $\R$, $\sim_0$ can be identified with the equivalence relation $x \sim y$ when $x-y\in \Z$ and $(X^u(\Per)/{\sim_0}, \tilde{g})$ can be identified with the original system $(S^1, g)$. In fact the results in this example generalize to the case when $g: Y \rightarrow Y$ is a local homeomorphism, in which case $(X^u(\Per)/{\sim_0}, \tilde{g})=(Y, g)$, see \cite{DeeYas} for details. 
\end{example}

\begin{example}
The relations $\sim_0$ and $\sim_1$ associated to the $aab/ab$-solenoid defined in Example \ref{aababSolEx} can be viewed as in Figures \ref{Figure-aab/ab-Tilde_0} and \ref{Figure-aab/ab-Tilde_1}. 

The details for $G_0(\Per)$ (i.e., $\sim_0$), which is illustrated in Figure \ref{Figure-aab/ab-Tilde_0}, are as follows. We take $\Per$ to be the set containing the fixed point associated to the wedge point (that is, $(p, p, p, \ldots)$ see Figure \ref{Figure-aab/ab-PreSolenoid}) and then have that $X^u(\Per)$ is homeomorphic to the real line. Intervals labelled with $a$ (resp. $b$) are mapped by $\pi_0=g\circ q$ to the outer (resp. inner) circle in $Y$. Identifying the endpoints of these intervals as $\Z$, we have that two integer points are equivalent if and only if the intervals to the left and right are labelled the same. While non-integer points are equivalent if and only if they are in intervals with the same label and their difference is in $\Z$. 

Figure \ref{Figure-aab/ab-Tilde_1} illustrates $G_1(\Per)$ in a similar way. The reader can find more details in both cases in \cite{DeeYas}.

\begin{figure}[h]
\begin{tikzpicture}
\draw [thick] (-4.5,0) -- (8.5,0);
\draw (-4,-.2) -- (-4, .2);
\draw (-2,-.2) -- (-2, .2);
\draw [ultra thick] (0,-.2) -- (0, .2);
\draw (2,-.2) -- (2, .2);
\draw (4,-.2) -- (4, .2);
\draw (6,-.2) -- (6, .2);
\draw (8,-.2) -- (8, .2);
\node at (-4.5,.25){$\ldots$};
\node at (-3,.25){$a$};
\node at (-1,.25){$b$};
\node at (0,-.5){$\mathbf{p}$};
\node at (1,.25){$a$};
\node at (2,-.5){$\mathbf{q}$};
\node at (3,.25){$a$};
\node at (5,.25){$b$};
\node at (7,.25){$a$};
\node at (8.5,.25){$\ldots$};
\end{tikzpicture}
\caption{$\sim_0$ for $aab/ab$ solenoid}
\label{Figure-aab/ab-Tilde_0}
\end{figure}
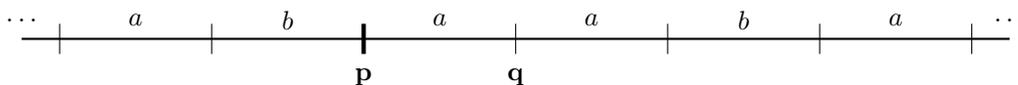

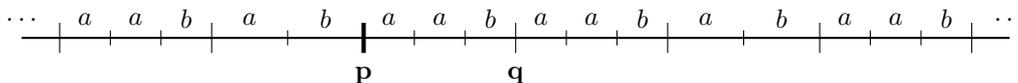
\begin{figure}[h]
\begin{tikzpicture}
\draw [thick] (-4.5,0) -- (8.5,0);
\draw (-4,-.2) -- (-4, .2);
\draw (-10/3,-.1) -- (-10/3, .1);
\draw (-8/3,-.1) -- (-8/3, .1);
\draw (-2,-.2) -- (-2, .2);
\draw (-1,-.1) -- (-1, .1);
\draw [ultra thick] (0,-.2) -- (0, .2);
\draw (2/3,-.1) -- (2/3, .1);
\draw (4/3,-.1) -- (4/3, .1);
\draw (2,-.2) -- (2, .2);
\draw (8/3,-.1) -- (8/3, .1);
\draw (10/3,-.1) -- (10/3, .1);
\draw (4,-.2) -- (4, .2);
\draw (5,-.1) -- (5, .1);
\draw (6,-.2) -- (6, .2);
\draw (20/3,-.1) -- (20/3, .1);
\draw (22/3,-.1) -- (22/3, .1);
\draw (8,-.2) -- (8, .2);
\node at (-4.5,.25){$\ldots$};
\node at (-11/3,.25){$a$};
\node at (-3,.25){$a$};
\node at (-7/3,.25){$b$};
\node at (-1.5,.25){$a$};
\node at (-.5,.25){$b$};
\node at (0,-.5){$\mathbf{p}$};
\node at (1/3,.25){$a$};
\node at (1,.25){$a$};
\node at (5/3,.25){$b$};
\node at (2,-.5){$\mathbf{q}$};
\node at (7/3,.25){$a$};
\node at (3,.25){$a$};
\node at (11/3,.25){$b$};
\node at (4.5,.25){$a$};
\node at (5.5,.25){$b$};
\node at (19/3,.25){$a$};
\node at (7,.25){$a$};
\node at (23/3,.25){$b$};
\node at (8.5,.25){$\ldots$};
\end{tikzpicture}
\caption{$\sim_1$ for $aab/ab$ solenoid}
\label{Figure-aab/ab-Tilde_1}
\end{figure}

We will show that in general $G_0(\Per)$ is not closed in $G_1(\Per)$ using this example. Notice that the points $p$ and $q$ are equivalent with respect to $\sim_1$ but not with respect to $\sim_0$. However there is a sequence of points (namely $p+\frac{1}{n} \sim_0 q+\frac{1}{n}$ where $(p+\frac{1}{n}, q+\frac{1}{n})$ converge to $(p, q)$). It follows that $\sim_0$ is not closed in $\sim_1$. 

Continuing with this example, we discuss $(X^u(\Per)/{\sim_0}, \tilde{g})$. Notice that the map $g: Y \rightarrow Y$ is not a local homeomorphism. Furthermore, the discussion of $G_0(\Per)$ above shows that $X^u(\Per)/{\sim_0}$ is given as follows. The point $p \in Y$ splits into three non-Hausdorff points, denoted $ab, ba, aa$. This space is illustrated in Figure \ref{Figure-aab/ab-QuotientSpace}.  These points correspond to the three different $\sim_0$-equivalence classes for ``integer" points, as seen in Figure \ref{Figure-aab/ab-Tilde_0}. 	Open neighborhoods of the three points (that is, $ab, ba, aa$) are pictured in Figure \ref{Figure-aab/ab-OpenNeighborhoods}.

	\begin{figure}[h]
		\begin{tikzpicture}
		\draw (0,0) circle [radius=3];
		\draw (0,-1) circle [radius=2];
		\draw[fill] (0,-2.6) circle [radius=0.1];
		\draw[fill] (0,-3) circle [radius=0.1];
		\draw[fill] (0,-3.4) circle [radius=0.1];
		\draw (0,.8) -- (0,1.2);
		\draw (-2.42487,1.4) -- (-2.77128,1.6);
		\draw (2.42487,1.4) -- (2.77128,1.6);
		\node [right] at (0.1,-2.7) {$ab$};
		\node [right] at (0.1,-3.2) {$ba$};
		\node [right] at (0.1,-3.6) {$aa$};
		\node [above] at ( -2.7, 1.6) {$q_1$};
		\node [above] at (0,1.1) {$q_2$};
		\node [above] at (2.7,1.6) {$q_3$};
	
		\end{tikzpicture}
		\caption{$X^u(\Per)/{\sim_0}$ for $aab/ab$ solenoid.}
		\label{Figure-aab/ab-QuotientSpace}
	\end{figure}
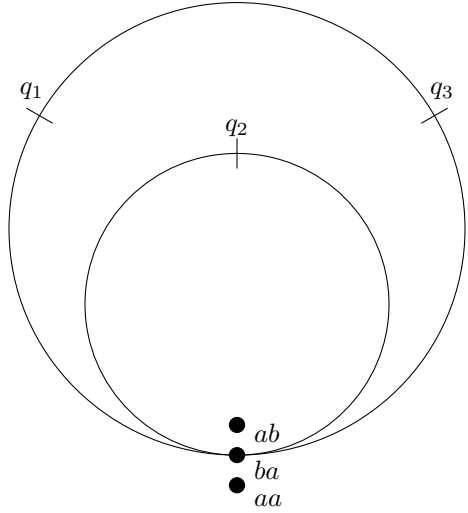

	\begin{figure}[h]
		\begin{tikzpicture}
		\draw [dashed] ($(0,0) + (225:3)$) arc (225:240:3);
		\draw [ultra thick] ($(0,0) + (240:3)$) arc (240:270:3);
		\draw [dashed] ($(0,0) + (270:3)$) arc (270:315:3);
		\draw [dashed] ($(0,-1) + (210:2)$) arc (210:270:2);
		\draw [ultra thick] ($(0,-1) + (270:2)$) arc (270:315:2);
		\draw [dashed] ($(0,-1) + (315:2)$) arc (315:330:2);
		\node [rotate=-30] at (-1.5,-2.61) {$($};
		\node [rotate=45] at (1.414,-2.414) {$)$};
		\draw (0,-2.6) circle [radius=0.1];
		\draw[fill] (0,-3) circle [radius=0.1];
		\draw (0,-3.4) circle [radius=0.1];
		\node [right] at (0.1,-2.7) {$ab$};
		\node [right] at (0.1,-3.2) {$ba$};
		\node [right] at (0.1,-3.6) {$aa$};
		\end{tikzpicture} \qquad
		\begin{tikzpicture}
		\draw [dashed] ($(0,0) + (225:3)$) arc (225:240:3);
		\draw [dashed] ($(0,0) + (300:3)$) arc (300:315:3);
		\draw [ultra thick] ($(0,0) + (240:3)$) arc (240:270:3);
		\draw [ultra thick] ($(0,0) + (270:3)$) arc (270:300:3);
		\draw [dashed] ($(0,-1) + (210:2)$) arc (210:270:2);
		\draw [dashed] ($(0,-1) + (270:2)$) arc (270:330:2);
		\node [rotate=-30] at (-1.5,-2.61) {$($};
		\node [rotate=30] at (1.5,-2.61) {$)$};
		\draw (0,-2.6) circle [radius=0.1];
		\draw[fill=white] (0,-3) circle [radius=0.1];
		\draw[fill] (0,-3.4) circle [radius=0.1];
		
		\node [right] at (0.1,-3.6) {$aa$};
		\node [right] at (0.1,-2.7) {$ab$};
		\node [right] at (0.1,-3.2) {$ba$};
	
		\end{tikzpicture} \qquad
		\begin{tikzpicture}
		\draw [dashed] ($(0,-1) + (210:2)$) arc (210:225:2);
		\draw [dashed] ($(0,0) + (300:3)$) arc (300:315:3);
		\draw [dashed] ($(0,0) + (225:3)$) arc (225:270:3);
		\draw [ultra thick] ($(0,0) + (270:3)$) arc (270:300:3);
		\draw [ultra thick] ($(0,-1) + (225:2)$) arc (225:270:2);
		\draw [dashed] ($(0,-1) + (270:2)$) arc (270:330:2);
		\node [rotate=-45] at (-1.414,-2.414) {$($};
		\node [rotate=30] at (1.5,-2.61) {$)$};
		\draw[fill] (0,-2.6) circle [radius=0.1];
		\draw[fill=white] (0,-3) circle [radius=0.1];
		\draw (0,-3.4) circle [radius=0.1];
		
		\node [right] at (0.1,-3.6) {$aa$};
		\node [right] at (0.1,-2.7) {$ab$};
		\node [right] at (0.1,-3.2) {$ba$};

		\end{tikzpicture}
		\caption{Open neighborhoods of the three non-Hausdorff points in $X^u(\Per)/{\sim_0}$ for the $aab/ab$ solenoid.}
		\label{Figure-aab/ab-OpenNeighborhoods}
	\end{figure}
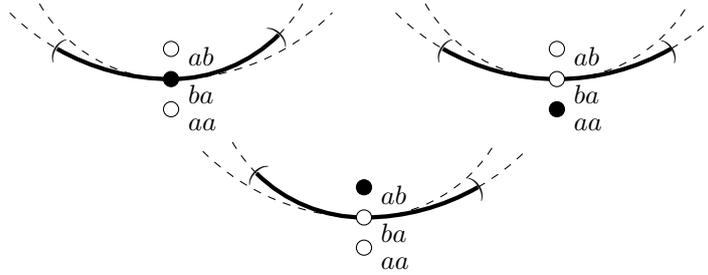

Next, we discuss the map $\tilde{g} : X^u(\Per)/{\sim_0} \rightarrow X^u(\Per)/{\sim_0}$. The map $\tilde{g}$ takes points labelled $ab$, $ba$, and $aa$ to $ba$. It takes the point labelled $q_1$ to $aa$, the point labelled $q_2$ to $ab$, and the point labelled $q_3$ to $ab$. The other points (i.e., the ones without labels) are maps in the same way as $g: Y \rightarrow Y$.

Some important properties to notice about $\tilde{g} : X^u(\Per)/{\sim_0} \rightarrow X^u(\Per)/{\sim_0}$ are the following:
\begin{enumerate}
\item $\tilde{g}$ is a local homeomorphisms, in particular it is open;
\item $\tilde{g}(r^{-1}(p))=\{ ba \}$ where $r : X^u(\Per)/{\sim_0} \rightarrow Y$ and $p$ is the wedge point of $Y$;
\item informally, $\tilde{g}$ is locally expanding.
\end{enumerate}
The importance of the second and third properties will be seen in Section \ref{ExpDSNonHausSec}. In particular, $\tilde{g}$ will be shown to be ``forward orbit expansive" in Section \ref{ExpDSNonHausSec}, which makes precise the third item in the list above.
\end{example}

\begin{example}
When $(Y, g)$ is the ab/ab-solenoid the map $g$ is not a local homeomorphism. We take the set $\Per$ to be the set containing the single element $(p,p, \ldots)$ where $p$ is the wedge point. Then, using a similar method to the one in the previous example, one can show that $X^u(\Per)$ is homeomorphic to $\R$, $X^u(\Per)/{\sim_0}$ is homeomorphic to the circle and $\tilde{g}$ is the two-fold cover of the circle. The main point of this example is that $X^u(\Per)/{\sim_0}$ can be Hausdorff even when $g$ is not a local homeomorphism. 
\end{example}

We restate \cite[Lemma 4.4]{DeeYas} here for ease of the reader; it will be used a number of times below.
\begin{lemma} \label{Lemma44inDeeYas}
There exists constant $K_0>0$ such that if $\xn$ and $\yn$ are in $X^u(P)$ and $x_i = y_i$ for $0\le i \le K_0$, then $\xn \sim_0 \yn$.
\end{lemma}

\begin{thm}
Suppose that $K_0$ is as in the previous lemma, $[\yn]_0\in X^u(\Per)/{\sim_0}$ and there exists $V$ open in $Y$ such that
\begin{enumerate}
\item $\pi_0(\yn)=y_0\in V$ and
\item $g^{K_0}|_V$ is injective.
\end{enumerate}
Then $r^{-1}(y_0)=\{ [\yn]_0\}$.
\end{thm}
\begin{proof}
Take $\delta>0$ such that $B(y_0, \delta) \subseteq V$.

Suppose that $\hat{\yn}=(y_0, \hat{y_1}, \hat{y_2}, \ldots) \in X^u(\Per)$. We must show $\yn \sim_0 \hat{\yn}$. By the previous lemma, 
\[
(g^{K_0}(y), \ldots, g(y), y, y_1, \ldots) \sim_0 (g^{K_0}(y), \ldots, g(y), y, \hat{y}_1, \ldots).
\]
By the definition of $\sim_0$, there exists open sets in $X^u(P)$, $U\subseteq X^u(\varphi^{K_0}(\yn), \delta)$ and $\hat{U}\subseteq X^u(\varphi^{K_0}(\hat{\yn}), \delta)$, such that 
\begin{enumerate}
\item $\varphi^{K_0}(\yn)=(g^{K_0}(y), \ldots, g(y), y, y_1, \ldots) \in U$;
\item $\varphi^{K_0}(\hat{\yn})=(g^{K_0}(y), \ldots, g(y), y, \hat{y}_1, \ldots) \in \hat{U}$;
\item $\pi_0(U)=\pi(\hat{U})$. 
\end{enumerate}
Since $\varphi|_{X^u(\Per)}$ is a homeomorphism, both $\varphi^{-K_0}(U)$ and $\varphi^{-K_0}(\hat{U})$ are open in $X^u(\Per)$. Furthermore, $\yn \in \varphi^{-K_0}(U)$ and $\hat{\yn} \in \varphi^{-K_0}(\hat{U})$. Notice that $\varphi^{-K_0}(U) \subseteq X^u(\yn, \delta)$ and $\varphi^{-K_0}(\hat{U}) \subseteq X^u(\hat{\yn}, \delta)$.

We show that $\pi_0(\varphi^{-K_0}(U)) = \pi_0(\varphi^{-K_0}(\hat{U}))$. Take $\zn \in \varphi^{-K_0}(U)$. Then 
\[
(g^{K_0}(z_0), \ldots, g(z_0), z_0, \ldots) \in U
\]
and by the third item in the list above, there exists $(\bar{z}_0, \bar{z}_1, \ldots) \in \hat{U} \subseteq X^u(\hat{\yn}, \delta)$ with $\bar{z}_0=g^{K_0}(z_0)$. Since $(\bar{z}_0, \bar{z}_1, \ldots) \in X^u(\varphi^{K_0}(\hat{\yn}), \delta)$, $\mathrm{d}_Y(z_{K_0}, y_0)< \delta$. However, $g^{K_0}$ is injective on $B(y_0, \delta) \subseteq V$ hence $z_{K_0}=z_0$ (i.e., $z_0$ is the unique preimage of $g^{K_0}(z_0)$ inside $V$). Hence $\pi_0(\zn)=\pi_0(z_{K_0}, z_{K_0+1}, \ldots)$ and the result follows.
\end{proof}

\begin{corollary}
If $(Y,g)$ is a Williams presolenoid, then $r$ is one-to-one on a dense open set of $X^u(\Per)/{\sim_0}$.
\end{corollary}
\begin{proof}
The conditions of the previous theorem hold for any non-branched point in $Y$. Since the set of branched points is closed and nowhere dense the result follows. (The reader can find the precise definition of branched point in \cite{Wil}. This is the only proof in the present paper that uses this term.)
\end{proof}

To summarize, we believe the results of this section along with the three examples gives sufficient motivation to consider the properties of the dynamical systems $(X^u(\Per))/{\sim_0}, \tilde{g})$ and its connection to $(X, \varphi)$ and $(Y, g)$. 

\section{Expansive dynamics in the non-Hausdorff setting} 
\label{ExpDSNonHausSec}
Since $X^u(\Per)/{\sim_0}$ is a non-Hausdorff space and we are interested in the dynamics on it, we discuss the general situation for expansive maps in the non-Hausdorff setting. Much of the general theory is based on work of Achigar, Artigue, and Monteverde \cite{MR3501269}, who consider the case of expansive homeomorphisms in the non-Hausdorff setting. The work in this section differs from \cite{MR3501269} in that we relax the homeomorphism assumption and consider a pair $(\tilde{Y}, \tilde{g})$ of a compact topological space $\tilde{Y}$ and a continuous surjective map $\tilde{g}: \tilde{Y} \rightarrow \tilde{Y}$; additional conditions on $\tilde{g}$ (such as expansiveness and being a local homeomorphism) will be explicitly stated as required. We emphasize that $\tilde{Y}$ need not be Hausdorff. It might be useful for the reader to have a copy of \cite{MR3501269} while reading the present section, so they compare results here to those in \cite{MR3501269}.

\subsection{Expansive maps on non-Hausdorff spaces}
\label{exponnonhaus}

\begin{definition} 
Suppose that $\mathcal{U}=\{ U_i \}_{i\in I}$ is an open cover of a topological space and $(x_n)_{n\in \N}$ and $(y_n)_{n\in \N}$ are sequences in the space. Then we write
\[
\{ x_n, y_n\}_{n\in \N} \prec \mathcal{U}
\]
if for each $n\in \N$, there exists $i_n\in I$ such that 
\[
x_n \hbox{ and }y_n \hbox{ are elements of }U_{i_n}.
\]
We use similar notation for two-sided sequences (that is, for $(x_n)_{n\in \Z}$ and $(y_n)_{n\in \Z}$)
\end{definition}

In \cite[Definition 2.1]{MR3501269}, the reader can find a notion of expansiveness for homeomorphisms of compact, non-Hausdorff spaces. Upon replacing the conditions of \cite{MR3501269} involving the $\Z$-action associated to a homeomorphism, with an $\N$-action we arrive at a completely analogous definition of orbit expansive maps that need not be invertible.

\begin{definition}
We say that $(\tilde{Y}, \tilde{g})$ is forward orbit expansive if there exists a finite open cover of $\tilde{Y}$, $\mathcal{U}=\{ U_i \}_{i=1}^l$, such that if $x$, $y$ are in $Y$ with
\[ \{ \tilde{g}^n(x), \tilde{g}^n(y) \}_{n\in \N} \prec \mathcal{U}, \]
then $x=y$.
\end{definition}

\begin{prop} 
Suppose that $(\tilde{Y}, \tilde{g})$ is forward orbit expansive and $(\tilde{X}, \tilde{\varphi})$ is the associated solenoid. Then $(\tilde{X}, \tilde{\varphi})$ is orbit expansive in the sense of \cite[Definition 2.1]{MR3501269}.
\end{prop}
\begin{proof}
Take an open cover of $\tilde{Y}$, $\mathcal{U}_{\tilde{Y}}=\{ U_i \}_{i=1}^l$, as in the definition of forward orbit expansive. Form 
\[ \mathcal{U}_X=\{ U_i\times \tilde{Y} \times \tilde{Y}\times \ldots \}_{i=1}^l, \]
which is an open cover of $\tilde{X}$. 

Suppose ${\bf x}=(x_0, x_1, \ldots)$ and ${\bf y}=(y_0, y_1, \ldots )$ are in $\tilde{X}$ with 
\[ \{ \varphi^n({\bf x}), \varphi^n({\bf y}) \}_{n\in \Z} \prec \mathcal{U}_X. \]
Then there exists $i_0 \in \{ 1, \ldots l \}$ such that ${\bf x}$ and ${\bf y}$ are in $U_{i_0}\times \tilde{Y} \times \tilde{Y}\times \ldots$. Hence, $x_0$ and $y_0$ are in $U_{i_0}$. Using the fact that 
$\tilde{\varphi}({\bf x})=( \tilde{g}(x_0), \tilde{g}(x_1), \ldots)$ and an induction argument give us 
\[ \{ \tilde{g}^n(x_0), \tilde{g}^n(y_0) \} \prec \mathcal{U}_Y.\]
The definition of forward orbit expansive then implies that $x_0=y_0$. Using the fact that $\tilde{\varphi}^{-1}({\bf x})=( x_1, x_2, \ldots)$ allows us to repeat the above argument to get $x_1=y_1$. The proof is then completed by showing that $x_n=y_n$ for each $n\in \N$ in a similar way; this implies that ${\bf x}={\bf y}$.
\end{proof}

\begin{prop} (compare with \cite[Proposition 2.5]{MR3501269})
Suppose that $(\tilde{Y}, \tilde{g})$ is forward orbit expansive, then $\tilde{Y}$ is $T_1$.
\end{prop}
\begin{proof}
Take an open cover of $\tilde{Y}$, $\mathcal{U}=\{ U_i \}_{i=1}^l$, as in the definition of forward orbit expansive. Suppose $x\neq y$ are in $\tilde{Y}$. Then, using the definition of forward orbit expansive, there exists $n\in \N$ and $i_x \neq i_y$ such that $\tilde{g}^n(x) \in U_{i_x}$ but $\tilde{g}^n(x)\not\in U_{i_y}$ and $\tilde{g}^n(y)\in U_{i_y}$ but $\tilde{g}^n(y)\not\in U_{i_x}$. Since $\tilde{g}$ is continuous, the sets $\tilde{g}^{-n}(U_{i_x})$ and $\tilde{g}^{-n}(U_{i_y})$ are open and the properties in the previous sentence imply that these sets are a $T_1$-separation of $x$ and $y$.
\end{proof}

\begin{prop} (compare with \cite[Theorem 2.7]{MR3501269})
Suppose that $Y$ is compact and Hausdorff, $(Y, g)$ is forward orbit expansive, and $g$ is an open map. Then $Y$ is metrizable and $(Y, g)$ is forward expansive.
\end{prop}
\begin{proof}
To begin, we note that $g$, in addition to being open, is also a closed map. Let $\mathcal{U}$ be an open cover of $Y$ as in the definition of forward orbit expansive. Given $U\in \mathcal{U}$ and $y\in U$, there exists $V_y$ open such that 
\[
y \in V_y \subseteq \overline{V}_y \subseteq U.
\]
Using this fact and the compactness of $Y$, there exists an open $\mathcal{V}=\{V_1, \ldots , V_m\}$ such that for each $V_i$ there exists $U \in \mathcal{U}$ with $\overline{V}_i \subseteq U$. This property and forward orbit expansiveness imply that 
\[
{\rm card}\left( \bigcap_{i\ge0} g^i(\overline{V}_{k_i}) \right) \le 1 
\]
for any $(k_i)_{i\ge0} \in \{ 1, \ldots m\}^{\N}$.

Now suppose that $y \in Y$ and $W$ is an open set with $y\in W$. Since $\mathcal{V}$ is a cover and $g$ is onto, there exists $(k_i)_{i\ge0} \in \{ 1, \ldots m\}^{\N}$ such 
\[
y \in  \bigcap_{i\ge0} g^i(V_{k_i}).
\]
It follows that $\bigcap_{i\ge0} g^i(\overline{V}_{k_i})=\{y\}$ and hence that
\[
W^c \cap \left( \bigcap_{i\ge0} g^i(\overline{V}_{k_i}) \right) = \emptyset.
\]
By the reformulation of compactness via the finite intersection property (this is also where we use the fact that $g$ is a closed map) there exists $N\in \N$ such that 
\[
W^c \cap \left( \bigcap_{i\ge0}^N g^i(\overline{V}_{k_i}) \right) = \emptyset
\]
and hence $\bigcap_{i\ge0}^N g^i(V_{k_i}) \subseteq W$.

It now follows that the collection 
\[
\left\{ \bigcap_{i\ge0}^N g^i(V_{k_i}) \mid N\in \N \hbox{ and }k_i \in \{ 1, \ldots, m\} \right\}
\]
is a basis for the topology on $Y$. Since this collection is countable (and $Y$ is compact and Hausdorff), the topology on $Y$ is metrizable.

For the second part of the theorem, again let $\mathcal{U}$ be an open cover of $Y$ as in the definition of forward orbit expansive. Fix a metric on $Y$ that induces the given topology. Let $\delta>0$ be the Lebesgue number of $\mathcal{U}$ with respect to $d$. One can show that $\delta$ is an expansive constant for $(Y, g)$; the details are omitted.
\end{proof}

\begin{prop}(compare with \cite[Proposition 2.14]{MR3501269}) \label{propPowerIsExpansive}
Suppose that $(\tilde{Y}, \tilde{g})$ is forward orbit expansive, and $\tilde{g}$ is an open map. Then, for each $n\in \N$, $\tilde{g}^n$ is forward orbit expansive.
\end{prop}
\begin{proof}
Take an open cover of $\tilde{Y}$, $\mathcal{U}=\{ U_i \}_{i=1}^l$, as in the definition of forward orbit expansive for $\tilde{g}$. Using the openness of the map $\tilde{g}$ we have that 
\[
\mathcal{U}_n= \{ U_{i_1} \cap \tilde{g}(U_{i_2}) \cap \ldots \cap \tilde{g}^{n-1}(U_{i_{n-1}}) \mid i_1, i_2, \ldots , i_{n-1} \in \{ 1, \ldots, l\} \}
\]
is an open cover of $\tilde{Y}$. Moreover, the fact that $\mathcal{U}$ is forward orbit expansive for $\tilde{g}$ implies that $\mathcal{U}_n$ is forward orbit expansive for $\tilde{g}^{n}$; the details are omitted.
\end{proof}

\begin{prop}
Suppose that $(\tilde{Y}, \tilde{g})$ is forward orbit expansive with $\tilde{g}$ open, $(\tilde{Y}_{{\rm Haus}}, \tilde{g}_{{\rm Haus}})$ denotes its Hausdorffization (see \cite{munsterthe}), and $\tilde{r}: \tilde{Y} \rightarrow \tilde{Y}_{{\rm Haus}}$ denotes the natural map. Furthermore suppose that there exists $L\in \N$ such that for each $y\in  \tilde{Y}_{{\rm Haus}}$, $\tilde{g}^L(\tilde{r}^{-1}(y))$ is a singleton. Then $\tilde{Y}$ is locally Hausdorff.
\end{prop}
\begin{proof}
By Proposition \ref{propPowerIsExpansive} and replacing $\tilde{g}$ with $\tilde{g}^L$, we can assume that $L=1$. Take $\mathcal{U}=\{ U_i \}_{i=1}^l$, as in the definition of forward orbit expansive and $x \in \tilde{Y}$. Let $U=U_{i_0}$ where $x\in U_{i_0}$. 

We will show that $U$ with the subspace topology is Hausdorff. Suppose $y_1$ and $y_2$ in $U$ cannot be separated. We will show that $y_1=y_2$. Since $\tilde{g}$ is continuous, $\tilde{g}(y_1)$ and $\tilde{g}(y_2)$ cannot be separated, but then $\tilde{g}(y_1)=\tilde{g}(y_2)$ (since $L=1$). Hence $\{ \tilde{g}^n(x), \tilde{g}^n(y) \} \prec \mathcal{U}$ and by the definition of forward expansive, $y_1=y_2$ as required.
\end{proof}

The next example shows that an additional condition (e.g., the assumption that $\tilde{g}^L(\tilde{r}^{-1}(y))$ is a singleton in the previous result) is required to ensure the solenoid is Hausdorff.

\begin{example}
Let $\tilde{Y}$ be the unit circle in $\C$ with two ``1"s, see Figure \ref{CircleTwoOnes}.

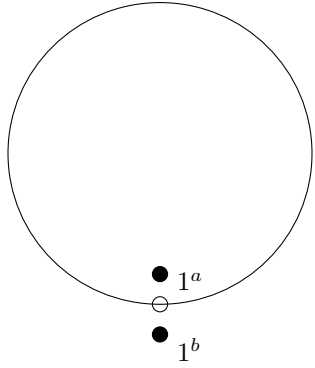
\begin{figure}[h]
		\begin{tikzpicture}
		\draw (0,-1) circle [radius=2];
		\draw[fill] (0,-2.6) circle [radius=0.1];
		\draw (0,-3) circle [radius=0.1];
		\draw[fill] (0,-3.4) circle [radius=0.1];
		\node [right] at (0.1,-2.7) {$1^a$};
		\node [right] at (0.1,-3.6) {$1^b$};
		\end{tikzpicture}
\caption{Circle with two 1s pre-solenoid}
\label{CircleTwoOnes}
\end{figure}
We define $\tilde{g}: \tilde{Y} \rightarrow \tilde{Y}$ to be the two fold covering map for points other than $-1$, $1^a$ and $1^b$; for those points we define
\[
-1 \mapsto 1^a \hbox{ and }1^a\mapsto 1^b \hbox{ and }1^b \mapsto 1^a.
\]
One can check that $\tilde{g}$ is continuous, onto, and open. Furthermore, the solenoid associated to $(\tilde{Y}, \tilde{g})$ is not Hausdorff since the points 
\[
(1^a, 1^b, 1^a, \ldots ) \hbox{ and }(1^b, 1^a, 1^b, \ldots)
\]
cannot be separated. 
\end{example}

\begin{thm} \label{PreNonHausSolHaus}
Suppose that $(\tilde{Y}, \tilde{g})$ is forward orbit expansive with $\tilde{g}$ open (and continuous and onto as usual in this section) and there exists $(Y, g)$ and $r: \tilde{Y} \rightarrow Y$ with 
\begin{enumerate}
\item $Y$ compact and Hausdorff;
\item $g$ is continuous and onto;
\item $r$ continuous, onto, and $r\circ \tilde{g}=g \circ r$;
\item there exists $L\in \N$ such that for each $y\in  Y$, $g^L(r^{-1}(y))$ is a singleton.
\end{enumerate}
Then the solenoids associated to $(\tilde{Y}, \tilde{g})$ and $(Y, g)$ are conjugate. In particular, the solenoid associated to $(\tilde{Y}, \tilde{g})$ is Hausdorff.
\end{thm}
\begin{proof}
Following \cite{Wil}, we will define a shift equivalence between $(\tilde{Y}, \tilde{g})$ and $(Y, g)$. In addition to the map $r: \tilde{Y} \rightarrow Y$ in the statement of the theorem, we have the map $s: Y \rightarrow \tilde{Y}$ defined via
\[ 
y \mapsto \tilde{g}^L(r^{-1}(y))
\]
where we have abused notation in the sense that $\tilde{g}^L(r^{-1}(y))$ is really a set that contains a single element. The map $s$ is onto since $r$ and $\tilde{g}$ are onto. One can also show that $s$ is continuous.

To show that $r$ and $s$ define a shift equivalence, we must show that 
\begin{enumerate}
\item $r \circ \tilde{g} = g\circ r$;
\item $s\circ g = \tilde{g} \circ s$;
\item $s\circ r=\tilde{g}^L$;
\item $r\circ s=g^L$.
\end{enumerate}
The first item is an assumption of the theorem. The fourth is similar to the third so only the proofs of the second and third will be considered in detail. 

To see that $s\circ g = \tilde{g} \circ s$, we have that
\[ (s\circ g)(y)=\tilde{g}^L(r^{-1}(g(y))) \hbox{ and }(\tilde{g}\circ s)(y)=\tilde{g}^{L+1}(r^{-1}(y)). \]
Using $r\circ \tilde{g}=g \circ r$, it follows that $\tilde{g}(r^{-1}(y)) \subseteq r^{-1}(g(y))$. Apply $\tilde{g}^L$ to both sides, we find that $\tilde{g}^L(r^{-1}(g(y)))$ is a singleton by assumption. Hence $\tilde{g}^{L+1}(r^{-1}(y))$ is the same singleton because it is non-empty and contained in the singleton set $\tilde{g}^L(r^{-1}(g(y)))$. 

To see that $s\circ r=\tilde{g}^L$, we have
\[ (s\circ r)(\tilde{y})=\tilde{g}^L(r^{-1}(r(\tilde{y}))). \]
Now, $\tilde{y}\in r^{-1}(r(y))$, so $\tilde{g}^L(\tilde{y})\in \tilde{g}^L(r^{-1}(r(\tilde{y})))$. The set $\tilde{g}^L(r^{-1}(r(\tilde{y})))$ is a singleton, so it must be equal to $\{ \tilde{g}^L(\tilde{y}) \}$ as required.

Using work in \cite{Wil}, it follows that the map $S: X \rightarrow \tilde{X}$ defined via
\[
(y_0, y_1, y_2, \ldots) \mapsto (s(y_0), s(y_1), s(y_2), \ldots )
\]
is a conjugacy from $(X, \varphi)$ (the solenoid associated to $(Y, g)$) to $(\tilde{X}, \tilde{\varphi})$ (the solenoid associated to $(\tilde{Y}, \tilde{g})$). Its inverse is given by $R : \tilde{X} \rightarrow X$ defined via
\[
(\tilde{y}_0, \tilde{y}_1, \tilde{y}_2, \ldots) \mapsto (r(\tilde{y}_0), r(\tilde{y}_1), r(\tilde{y}_2), \ldots ).
\] 
For the final part of the theorem, $X$ is Hausdorff because $Y$ is Hausdorff. Hence, the solenoid associated to $\tilde{X}$ is Hausdorff as well because $X$ and $\tilde{X}$ are homeomorphic.
\end{proof}

\subsection{The dynamics on the locally Hausdorff quotient space}

We now come to the main application of Subsection \ref{exponnonhaus} on Wieler-Smale spaces.

\begin{thm} 
\label{WielerNonHausOrbExp}
The dynamical system $(X^u(P)/{\sim_0}, \tilde{g})$ is forward orbit expansive. Moreover, $\tilde{g}$ is a local homeomorphism and there exists $L\in \N$ such that for each $y\in  \tilde{Y}_{{\rm Haus}}$, $\tilde{g}^L(r^{-1}(y))$ is a singleton.
\end{thm}
\begin{proof}
Recall the constants $\beta>0$ and $0<\lambda<1$ in the definition of a Wieler presolenoid, see Definition \ref{WielerAxioms}. Using the nature of the local stable and unstable sets for a Wieler solenoid (see Definition \ref{WieSolenoid}) there exists $0<\delta< \frac{\beta}{2}$ such that if $\xn \neq \yn$ are in $X^u({\bf z}, \delta)$ (for some $\zn$) then $x_0\neq y_0$. Using the compactness of $X^u(P)/{\sim_0}$, we have an open cover of the form
\[
\mathcal{U}= \{ \pi_0(X^u(\zn_i, \delta)) \cap r^{-1}(B(w_j, \frac{\beta}{2})) \}_{i\in I, j\in J}.
\]
where $\zn_i \in X^u(P)$, $w_j\in Y$, $I$ and $J$ are finite index sets, and the fact that $\pi_0$ is an open map and $r$ is continuous ensures that the given sets are open.

Let $[\xn]_0 \neq [\yn]_0$ be in $X^u(P)/{\sim_0}$. We will find $N\in \N$ such that for each $i$, the two element set $\{ \tilde{g}^N([\xn]_0), \tilde{g}^N([\yn]_0) \}$ is not a subset of $\pi_0(X^u({\bf z}_i, \delta))\cap r^{-1}(B(w_j, \frac{\beta}{2}))$. 

We are done unless there exists $i_0$ such that 
\[ [\xn]_0 \hbox{ and } [\yn]_0 \hbox{ are both in }\pi_0(X^u({\bf z}_{i_0}, \delta)).
\]
Let $\xn$ and $\yn$ be points in $X^u({\bf z}_{i_0}, \delta)$ representing $[\xn]_0$ and $[\yn]_0$ respectively. By assumption, $x_0 \neq y_0$ and 
\[
\mathrm{d}_Y(x_K, y_K)< 2\delta < \beta
\]
where $K$ is the constant in Wieler's Axioms. Using the first of Wieler's Axioms, 
\[ 0<\mathrm{d}_Y(x_0, y_0) \le \gamma^K \mathrm{d}_Y(g^K(x_0), g^K(y_0)). \]
Using this inequality and the fact that $x_0 \neq y_0$, we have that $g^K(x_0) \neq g^K(y_0)$. If $\mathrm{d}_Y(g^K(x_0), g^K(y_0))> \beta$, then we stop. Otherwise, since 
\[
\mathrm{d}_Y(x_0, y_0)< 2\delta < \beta,
\]
we can again use the first of Wieler's Axioms. We get that
\[ \mathrm{d}_Y(g^K(x_0), g^K(y_0)) \le \gamma^{K} \mathrm{d}_Y(g^{2K}(x_0), g^{2K}(y_0)). \]
Using this inequality and the previous one, we obtain 
\[ 0<\mathrm{d}_Y(x_0, y_0) \le \gamma^{2K} \mathrm{d}_Y(g^{2K}(x_0), g^{2K}(y_0)).
\]
Using the fact that $0<\gamma<1$ and possibly repeating this process, there exists $N\in \N$ such that $d(g^N(x_0), g^N(y_0))> \beta$. Now,
\[
r(\tilde{g}^N([\xn]_0))=g^N(x_0) \hbox{ and }r(\tilde{g}^N([\yn]_0))=g^N(y_0).
\]
Therefore, $\tilde{g}^N([\xn]_0)$ and $\tilde{g}^N([\yn]_0)$ cannot both be in $r^{-1}(B(w_j, \frac{\beta}{2}))$ for any $j$. Thus, $\tilde{g}$ is forward orbit expansive.

For the second part of the theorem, it was already noted that $\tilde{g}$ is a local homeomorphism and the existence of the required $L$ follows from Lemma \ref{Lemma44inDeeYas}.
\end{proof}

\subsection{Inverse limit space associated to the spectrum}

The main goal of this section is to show that the inverse limit formed from $(X^u(P)/{\sim_0}, \tilde{g})$ also gives the Wieler solenoid associated to $(Y, g)$. This is perhaps somewhat surprising in light of the fact that $X^u(P)/{\sim_0}$ is often non-Hausdorff. 

To fix notation, $(Y, g)$ is assumed to satisfy Wieler's axioms, $(X, \varphi)$ is the associated solenoid and $(X^u(P)/{\sim_0}, \tilde{g})$ is as in the previous section. We let
\[
\tilde{X}:= \varprojlim (X^u(P)/{\sim_0}, \tilde{g}) = \{ (\tilde{y}_n)_{n\in \N} = (\tilde{y}_0, \tilde{y}_1, \tilde{y}_2, \ldots ) \: | \: \tilde{g}(\tilde{y}_{i+1})=\tilde{y}_i \hbox{ for each }i\ge0 \}
\]
with the map $\tilde{\varphi}: \tilde{X} \rightarrow \tilde{X}$ be defined via
\[
\tilde{\varphi}(\tilde{y}_0, \tilde{y}_1, \tilde{y}_2, \ldots ) = (\tilde{g}(\tilde{y}_0), \tilde{g}(\tilde{y}_1), \tilde{g}(\tilde{y}_2), \ldots) = (\tilde{g}(\tilde{y}_0), \tilde{y}_0, \tilde{y}_1, \ldots ).
\] 
Again we will make use of Lemma 4.4 in \cite{DeeYas}, which we restate here for the reader.

We have the two maps. The first one was defined in \cite{DeeYas}; it is 
\[
r: X^u(P)/{\sim_0} \rightarrow Y \hbox{ defined via } [\xn]_0 \mapsto x_0
\]
where we note that the definition of $\sim_0$ implies that $r$ is well-defined. Furthermore, it is continuous and surjective. 

The second map is 
\[
s: Y \rightarrow X^u(P)/{\sim_0} \hbox{ defined via }y \mapsto [g^{K_0}(y), g^{K_0-1}(y), \ldots , g(y), y,  \ldots ]_0
\]
where we note that the Lemma 4.4 in \cite{DeeYas} implies that $s$ is well-defined.

\begin{thm} \label{InverseLimitXUP}
Using the notation in the previous paragraphs, the maps $r$ and $s$ define a shift equivalence. That is, they satisfy 
\[
r \circ \tilde{g}=g\circ r, s \circ g =g \circ s,  r \circ s = g^{K_0} \hbox{ and } s \circ r = \tilde{g}^{K_0}.
\]
Moreover, the map $S: X \rightarrow \tilde{X}$ defined via
\[
(y_0, y_1, y_2, \ldots) \mapsto (s(y_0), s(y_1), s(y_2), \ldots )
\]
is a conjugacy from $(X, \varphi)$ to $(\tilde{X}, \tilde{\varphi})$.
\end{thm}
\begin{proof}
This follows from the statement and proof of Theorem \ref{PreNonHausSolHaus}. We note that the assumptions of Theorem \ref{PreNonHausSolHaus} hold by Theorem \ref{WielerNonHausOrbExp}. Also, the reader can check that the map $s$ defined just before the statement of the current theorem is equal to the map $s$ consider in Theorem \ref{PreNonHausSolHaus}.
\end{proof}

\begin{corollary}
If $(X, \varphi)$ is an irreducible Smale space with totally disconnected stable sets, then there exists $(\tilde{Y}, \tilde{g})$ such that $\tilde{g}$ is a surjective local homeomorphism that is forward orbit expansive and $(X, \varphi)$ is conjugate the solenoid associated to $(\tilde{Y}, \tilde{g})$.
\end{corollary}
\begin{proof}
This follows from the previous result and Wieler's theorem.
\end{proof}

\subsection{The relationship between $(X, \varphi)$, $(X^u(\Per)/{\sim_0}, \tilde{g})$ and $(Y, g)$}

It follows from Theorem \ref{InverseLimitXUP} that there is a continuous surjection from $X$ to $X^u(P)/{\sim_0}$. However, we can write it is explicitly without reference to the conjugacy in Theorem \ref{InverseLimitXUP}:

\begin{thm}
Define $p : X \rightarrow X^u(P)/{\sim_0}$ via 
\[
x \mapsto [y]_0
\]
where $y\in X^u(P) \cap X^s(x, \frac{\epsilon_X}{2})$. Then $p$ is a continuous surjection and for each $z\in X^u(P)/{\sim_0}$, $p^{-1}(z)$ is a Cantor set.
\end{thm}

\begin{proof}
To begin, we must show that $p$ is well-defined. Suppose that $y'$ is another element in $X^u(P) \cap X^s(x, \frac{\epsilon_X}{2})$. Firstly, since both $y$ and $y'$ are in $X^s(x, \frac{\epsilon_X}{2})$ we have that $\pi_0(x)=\pi_0(y)=\pi_0(y')$. Moreover, properties of the bracket implies that the map $h : X^u(y, \epsilon_X) \rightarrow X^u(y', \epsilon_X)$ defined via
\[ z \mapsto [z, y'] \]
is well defined and that $\pi_0(z)=\pi_0(h(z))$ for each $z \in X^u(y, \epsilon_X)$. This implies that $[y]_0=[y']_0$.

An equivalent definition of $p$ is the following: Let $U(x, \frac{\epsilon_X}{2})$ denote the image of the bracket of the set $X^s(x, \frac{\epsilon_X}{2}) \times X^u(x, \frac{\epsilon_X}{2})$. Given $x \in X$ and $w \in X^u(P) \cap U(x, \frac{\epsilon_X}{2})$ then $p(x)=[x, w]$. To see this is the same as the previous definition, we need only check that $[x, w]$ is an element of $X^u(P) \cap X^s(x, \frac{\epsilon_X}{2})$ but this follows from the definitions of the bracket and the set $U(x, \frac{\epsilon_X}{2})$.

\end{proof}

\begin{remark}
If $g$ is a local homeomorphism satisfying Wieler's axioms, then $X^u(P)/{\sim_0}=Y$ and the map $p:X\to Y$ is a locally trivial bundle of Cantor sets. This fact follows from \cite[Theorem 3.12]{DGMW}. Already for Williams solenoids, local triviality of $p$ fails to hold in general \cite{FJ2}.

\end{remark}

Moving forward we will use an abuse of notation to identify sets of the form $X^s(x, \delta_1) \times X^u(x, \delta_2)$ with their image in $X$ under the bracket map. For example, using this convention, there would be no reference to the set $U(x, \frac{\epsilon_X}{2})$ in the proof of the previous theorem.
\begin{lemma} \label{stableSetNotImportant}
Suppose $x\in X$ and $0< \delta<\delta'< \epsilon'$. Then 
\[
p(X^s(x, \delta)\times X^u(x, \epsilon'))=p(X^s(x, \delta')\times X^u(x, \epsilon')).
\]
\end{lemma}
\begin{proof}
Let $z\in X^s(x, \delta')\times X^u(x, \epsilon')$. Then $[z, x] \in X^u(x, \epsilon') \subseteq X^s(x, \delta)\times X^u(x, \epsilon')$. Furthermore, taking $w\in X^u(P) \cap U(x, \frac{\epsilon_X}{2})$, we have that
\[ p(z)=[z, w] = [ [z, x], w] = p([z, x]) \]
as required.
\end{proof}

\begin{thm}
Using the notation in the past few paragraphs, we have that the following diagram commutes: 
\begin{center}
		$\begin{CD}
		X @>\varphi >> X \\
		@Vp VV @Vp VV \\
		X^u(P)/{\sim_0} @>\tilde{g} >> X^u(P)/{\sim_0} \\
		@Vr VV @Vr VV \\
		Y @>g>> Y
		\end{CD}$
	\end{center}
where $r$ is defined above (see Section \ref{tildeZeroSection}) and $r \circ p$ is equal to the projection map $ \pi_0 : X \rightarrow Y$.
\end{thm}
\begin{proof}
The result follows directly from the definitions of the relevant maps; the details are omitted.
\end{proof}

\begin{thm} \label{perPt}
The map $p$ induces a bijection between periodic points of $\varphi$ and periodic points of $\tilde{g}$ and hence $\varphi$ and $\tilde{g}$ have the same zeta function. Furthermore, if the set of periodic points with respect to $\varphi$ is dense in $X$, then the set of periodic points with respect to $\tilde{g}$ is dense in $X^u(P)/{\sim_0}$.
\end{thm}
\begin{proof}
The first statement using Theorem \ref{InverseLimitXUP} and applying the proof of \cite[Lemma 5.3]{Wil} to our situation. The second follows since $p$ is onto so the image of a dense set in $X$ under $p$ is dense in $X^u(P)/{\sim_0}$. 
\end{proof}

\begin{thm}
The map $r: X^u(P)/{\sim_0} \rightarrow Y$ induces a bijection between the periodic points of $\tilde{g}$ and the periodic points of $g$. Hence $\tilde{g}$ and $g$ have the same zeta function.
\end{thm}
\begin{proof}
This follows from Theorem \ref{InverseLimitXUP} and \cite[Theorem 5.2 Part (3)]{Wil}, see in particular the argument just before Lemma 5.3 on page 189 of \cite{Wil}.
\end{proof}

\begin{thm}
The following are equivalent 
\begin{enumerate}
\item $(X, \varphi)$ is mixing, 
\item $(X^u(P)/{\sim_0}, \tilde{g})$ is mixing,
\item $(Y, g)$ is mixing.
\end{enumerate}
\end{thm}
\begin{proof}
We only prove the case $(X, \varphi)$ is mixing implies $(X^u(P)/{\sim_0}, \tilde{g})$ is mixing in detail. Let $U$ and $V$ be non-empty open sets in $X^u(P)/{\sim_0}$. Then $p^{-1}(U)$ and $p^{-1}(V)$ are non-empty open sets in $X$ and since $(X, \varphi)$ is mixing, there exists $N\in \N$ such that
\[
\varphi^n(p^{-1}(U)) \cap p^{-1}(V) \neq \emptyset \hbox{ for each }n\ge N
\]
Then, using the fact that $p$ is onto,
\begin{align*}
p(\varphi^n(p^{-1}(U)) \cap p^{-1}(V)) & \subseteq (p(\varphi^n(p^{-1}(U)))) \cap p(p^{-1}(V)) \\
& \subseteq (\tilde{g}^n( p(p^{-1}(U)))) \cap p(p^{-1}(V)) \\
& = \tilde{g}^n(U) \cap V
\end{align*}
This implies the result.

Notice that we have only used the following properties: $p$ is onto, continuous and $\tilde{g} \circ p = p \circ \varphi$. Thus to see that $(X^u(P)/{\sim_0}, \tilde{g})$ is mixing implies that $(Y, g)$ is mixing, one replaces $p$ with $r$ in the previous argument.
\end{proof}

\section{Full projections}

\label{fullprojectioninfell}
The main goal of this section is to prove that $C^*(G_0(\Per))$ contains a full projection. This result is used in Section \ref{cpsectionone} to construct unital Cuntz-Pimsner models. In light of \cite{DeeGofYasFellAlgPaper}, the existence of a full projection in $C^*(G_0(\Per))$ puts restrictions on the type of Fell algebras that appear as $C^*(G_0(\Per))$.

\subsection{Dynamical results}

\begin{lemma} 
\label{gInvOpenDense}
Suppose $(X^u(P)/{\sim_0}, \tilde{g})$ is mixing and $U$ is a non-empty open set in $X^u(P)/{\sim_0}$ such that there exists $L\in \N$ that satisfies $\tilde{g}^L(U) \subseteq U$. Then $U$ is dense in $X^u(P)/{\sim_0}$.
\end{lemma}
\begin{proof}
The result will follow by showing that if $V$ is a non-empty open set in $X^u(P)/{\sim_0}$ then $U \cap V \neq \emptyset$. Since $\tilde{g}$ is mixing, there exists $N \in \N$ such that for each $n\ge N$, 
\[
\tilde{g}^n(U) \cap V \neq \emptyset
\]
Thus $ \emptyset \neq \tilde{g}^{N\cdot L}(U) \cap V  \subseteq U \cap V$.
\end{proof}

\begin{lemma} 
\label{locallyExpansiveOfgmap}
Suppose that $(X, \varphi)$ is mixing and $\emptyset \neq U \subseteq X^u(P)/{\sim_0}$ is open. Then there exists $\emptyset \neq V \subseteq U$ open and $N \in \N$ such that
$V \subseteq \tilde{g}^N(V)$.
\end{lemma}
\begin{proof}
Since periodic points of $(X, \varphi)$ are dense, there exists $x \in p^{-1}(U)$ and $N \in \N$ such that $\varphi^N(x)=x$. Since $\varphi$ and $p$ are both continuous, there exists $\delta>0$ and $0<\delta'< \epsilon'$ such that $\varphi^N(X^u(x, \delta)) \subseteq X^u(x, \delta')$ and $p(X^s(x, \delta) \times X^u(x, \delta))\subseteq U$.

The set $X^s(x, \delta) \times X^u(x, \delta)$ is open in $X$ and so $p(X^s(x, \delta) \times X^u(x, \delta))$ is open in $X^u(P)/{\sim_0}$. We also have that $X^u(x, \delta) \subseteq \varphi^N(X^u(x, \delta))$. 

By Lemma \ref{stableSetNotImportant} and the fact that $\varphi$ contracts in the local stable direction 
\[
p(\varphi^N(X^s(x, \delta) \times X^u(x, \delta)))=p(X^s(x, \delta) \times \varphi^N(X^u(x, \delta)))
\]
and using $X^u(x, \delta) \subseteq \varphi^N(X^u(x, \delta))$, we have that 
\[
p(X^s(x, \delta) \times X^u(x, \delta)) \subseteq  p(\varphi^N(X^s(x, \delta) \times X^u(x, \delta)))=\tilde{g}^N(p(X^s(x, \delta) \times X^u(x, \delta)))
\]
\end{proof}

\begin{thm} \label{gExpOnXU}
Suppose that $(X, \varphi)$ is mixing and $U \subseteq X^u(P)/{\sim_0}$ is a non-empty, open set. Then there exists $N\in \N$ such that $\tilde{g}^N(U) = X^u(P)/{\sim_0}$.
\end{thm}
\begin{proof}
By Lemma \ref{locallyExpansiveOfgmap}, there exists a nonempty open set $V$ such that $V \subseteq U$ and $V\subseteq \tilde{g}^K(V)$. Moreover, the $V$ constructed in Lemma \ref{locallyExpansiveOfgmap} depended on $\delta>0$ and so we denote the set associated to $\delta>0$ as $V_{\delta}$. In more detail, $V_{\delta}= p (X^s(z, \delta) \times X^u(z, \delta))$ where $z$ is a periodic point of period $K$ in $U$.

Form
\[ G_{\delta}= \bigcup_{n\in \N} \tilde{g}^{n \cdot K}(V_{\delta}). \]
Notice that 
\[ G_{\delta}=p\left(\bigcup_{n\in \N} \varphi^{n\cdot K}(X^s(z, \delta) \times X^u(z, \delta)) \right) \]
and that since $(X, \varphi)$ is mixing
\[
\bigcup_{n\in \N} \varphi^{n\cdot K}(X^s(z, \delta) \times X^u(z, \delta))
\]
is dense in $X$. Since $p$ is onto, it follows that $G_{\delta}$ is dense for each valid $\delta>0$. In particular, $G_{\delta/2}$ is dense.

Next, suppose that $\tilde{y}\in X^u(P)/{\sim_0}$ is a limit point of the set $G_{\delta/2}$. We show $\tilde{y} \in G_{\delta}$. From this, it follows that $G_{\delta}=X^u(P)/{\sim_0}$. Using the compactness of $X$ and the fact that
\[
\bigcup_{n\in \N} \varphi^{n\cdot K}(X^s(z, \delta) \times X^u(z, \delta))
\]
is dense in $X$, there exists a sequence $(x_k)_{k\in \N}$ in $X$  converging to $x \in X$ such that 
\[ x_k \in \bigcup_{n\in \N} \varphi^{n\cdot K}(X^s(z, \delta/2) \times X^u(z, \delta/2)) \]
for each $n$ and $p(x)=\tilde{y}$. There exists $N \in \N$ such that $x_k \in X^s(x, \frac{\delta}{2}) \times X^u(x, \frac{\delta}{2})$ for $k \ge N$. In particular, $[x, x_N]$ is well-defined and is an element in $X^u(x_N, \frac{\delta}{2})$. Lemma \ref{stableSetNotImportant} implies that $p([x, x_N])=p(x)=\tilde{y}$. Since $x_N \in \bigcup_{n\in \N} \varphi^{n\cdot K}(X^s(z, \delta/2) \times X^u(z, \delta/2))$ there exists $l \in \N$ such that 
\[
x_N \in \varphi^{l \cdot K}(X^s(z, \delta/2) \times X^u(z, \delta/2))
\]
which implies that 
\[ \varphi^{-l \cdot K}(x_N) \in X^s(z, \delta/2) \times X^u(z, \delta/2). 
\]
Since $[x, x_N]$ is an element in $X^u(x_N, \frac{\delta}{2})$, we have that
\[
\varphi^{-l \cdot K}(x) \in X^u(\varphi^{-l\cdot K}(x_N), \frac{\delta}{2}).
\]
The triangle inequality and the previous two statements imply that 
\[
\varphi^{-l \cdot K}(x) \in X^s(z, \delta) \times X^u(z, \delta). 
\]
and hence that
\[ 
x \in \varphi^{l \cdot K}(X^s(z, \delta) \times X^u(z, \delta)) \hbox{ and } \tilde{y}=p(x) \in G_{\delta}. \]

In summary, we have shown that $G_{\delta}=X^u(P)/{\sim_0}$. Using this along with the fact that $X^u(P)/{\sim_0}$ is compact and $V_{\delta} \subseteq \tilde{g}^K(V_{\delta})$ imply that there exists $N\in \N$ such that $\tilde{g}^N(V_{\delta})=X^u(P)/{\sim_0}$. The required result follows since $V_{\delta} \subseteq U$.

\end{proof}

\subsection{Existence of a full projection}
\begin{lemma}
Let $V$ be a basic set of $G_0(\Per)$. Then, there exists $N\in \N$ such that for any $f \in C_c(G_0(\Per))$ supported in another basic set and nonzero on $V$, we have that the following two properties hold: 
\begin{enumerate}
\item For each $x \in X^u(P)$ there exists $(\hat{x}, \hat{w}) \in {\rm supp}(\alpha^N(f))$ such that $x \sim_0 \hat{x}$;
\item likewise for each $b\in X^u(P)$ there exists $(\hat{a}, \hat{b}) \in {\rm supp}(\alpha^N(f))$ such that $b \sim_0 \hat{b}$.
\end{enumerate}
\end{lemma}
\begin{proof}
This follows by applying Theorem \ref{gExpOnXU} to the open set $s(V)$ and $r(V)$ and taking the max of the two $N$s.
\end{proof}

\begin{lemma}
Let $V$ be a basic set of $G_0(\Per)$. Then, there exists $N\in \N$ such that for any $f \in C_c(G_0(\Per))$ supported in another basic set with $||f|_V||>0$, we have that for each $k \in C_c(G_0(\Per))$ supported in yet another basic set, there exists $f_1$, $f_2 \in C_c(G_0(\Per))$ such that
\[
f_1 \alpha^K(f) f_2 = k \hbox{ and } ||f_1||\leq ||k|| , ||f_2||= \frac{1}{||f|_V||}
\]
Moreover, in this case, we have that $\overline{C^*(G_0(\Per)) \alpha^N(f) C^*(G_0(\Per))}=C^*(G_0(\Per))$. 
\end{lemma}
\begin{proof}
Let $A=C^*(G_0(\Per))$. The second part of the theorem follows from the first part. To see this note that $\overline{A \alpha^N(f) A}$ denotes the closed linear span of $A \alpha^N(f) A$ and that closed linear span of the set of functions of compact support with support contained in a single basic set is $A$. 

We now prove the first part. Recall that $q: X^u(\Per) \rightarrow X^u(\Per)/{\sim_0}$ denotes the quotient map. A basic set, $U$, is of the form
\[
\{ (h_U(x), x) \mid x \in X^u(z, \delta) \}
\]
where $z\in X^u(\Per)$, $\delta>0$ is small, and $h_U$ is a local homeomorphism onto its image. Because we are dealing with a very specific groupoid, we have that $h_U= (q|_{r(U)})^{-1} \circ q|_{s(U)}$. This in particular holds for the basic set $U_k$ that contains the support of $k$. We likewise let $U_f$ denote the basic set containing the support of $f$. By the previous lemma, there exists $N\in \N$ such that
\begin{enumerate}
\item $q(s(\varphi^N \times \varphi^N(U_f)))=X^u(\Per)/{\sim_0}$ and
\item $q(r(\varphi^N \times \varphi^N(U_f)))=X^u(\Per)/{\sim_0}$.
\end{enumerate}
It follows that we have $V_1 \subseteq s(\varphi^N \times \varphi^N(U_f))$ and $V_2 \subseteq r(\varphi^N \times \varphi^N(U_f))$ such that 
\begin{enumerate}
\item $q|_{s(U_k)}$ is a homeomorphism from $s(U_k)$ to $V_1$ and
\item $q|_{V_2}$ is a homeomorphism from $V_2$ to $r(U_k)$.
\end{enumerate}
Since $q(r(U_k))=q(s(U_k))$ and all the relevant maps are homeomorphism (because the domains have been restricted) we have
\[
(q|_{r(U_k)})^{-1} \circ q|_{s(U_k)} = (q|_{r(U_k)})^{-1} \circ q|_{V_2} \circ (q|_{V_2})^{-1} \circ q|_{V_1} \circ (q|_{V_1})^{-1} \circ q|_{s(U)}.
\]

Using this setup we can define $f_1$ and $f_2$ as follows. For $y\in s(U_k)$, let
\[
f_1( ((q|_{V_1})^{-1} \circ q|_{s(U_k)})(y), y)=k(h_{U_k}(y), y)
\]
and otherwise define $f_1$ to be zero. For $z\in V_2$, let
\[
f_2( ((q|_{r(U_k)})^{-1} \circ q|_{V_2})(z), z) = \frac{1}{ f(\varphi^{-N}(z),  \varphi^{-N}( ((q|_{V_1})^{-1} \circ q|_{V_2})(z)))}
\]
and then extend $f_2$ so that its support is contained in a slightly large basic set. We note that $f_2$ is well-defined because $f$ is non-zero on $V$. A short computation using the convolution product shows that $f_1$ and $f_2$ satisfy requirements in the statement of the theorem.
\end{proof}

\begin{thm}
\label{existfull}
Suppose that $a \in A=C^*(G_0(\Per))$ is nonzero, then there exists $N\in \N$ such that $\alpha^N(a)$ is full in A. In particular, $C^*(G_0(\Per))$ contains a full projection.
\end{thm}
\begin{proof}
Let $a\in A=C^*(G_0(\Per))$ be nonzero. Without loss of generality we can and will assume that $||a||=1$. 

There exists a basic set $U \in \mathcal{G}$ with the following property. If $c \in C_c(\mathcal{G})$ with $||c||=1$ and 
\[ ||a-c||< \frac{1}{4} \]
then $|c(x,y)|>\frac{1}{2}$ for each $(x,y) \in U$. 

Take a basic set such that $V\subseteq \overline{V} \subseteq U$ and also take $\tilde{f}$ a continuous bump function with support contained in $U$ that is one on $V$. Let $N\in \N$ be as in the previous lemma where $N$ depends only on $V$. 

Let $0\neq k \in C_c(\mathcal{G})$ be supported in a basic set and $\epsilon>0$. Take $b \in C_c(\mathcal{G})$ such that 
\[ ||a-b||< {\rm min}\left\{ \frac{1}{4} , \frac{\epsilon}{||k||} \right\} \hbox{ and } ||b||=1 \]
By construction, the previous lemma can be applied to $\tilde{f} b \tilde{f}$. We obtain $f_1$ and $f_2$ such that
\[
f_1 \alpha^N(\tilde{f} b \tilde{f} ) f_2 =k
\]
Therefore
\begin{align*}
||f_1 \alpha^N(\tilde{f}) \alpha^N(a) \alpha^N(\tilde{f}) f_2 -k|| 
& = ||f_1 \alpha^N(\tilde{f}) \alpha^N(a) \alpha^N(\tilde{f}) f_2 - f_1 \alpha^N(\tilde{f}) \alpha^N(b) \alpha^N( \tilde{f} ) f_2 || \\
& \le ||f_1|| || \alpha^N(\tilde{f})|| || a - b|| || \alpha^N(\tilde{f})|| ||f_2|| \\
& < \epsilon
\end{align*}
where we have used the estimates for $||f_i||$ from the previous lemma and the fact that $||\tilde{f}||=1$.
In summary we have shown that the set of compactly supported functions with support in a single basic set is contained in $\overline{A \alpha^N(a) A}$ and this implies the result.

For the second part, by the main result of \cite{DGY}, $C^*(G_0(\Per))$ contains a non-zero projection. Using the first part it therefore contains a full projection.
\end{proof}

\begin{corollary}
\label{existefullforcom}
There exists a projection $p\in C_c(G_0(\Per))$ which is full in $C^*(G_0(\Per))$. In particular, the subalgebra $A_p:=pC^*(G_0(\Per))p$ is a unital Fell algebra with spectrum $X^u(\Per)/{\sim_0}$ and the dense unital subalgebra $pC_c(G_0(\Per))p\subseteq A_p$ is closed under holomorphic functional calculus.
\end{corollary}

\begin{proof}
By Theorem \ref{existfull}, there exists a full projection $p_0\in C^*(G_0(\Per))$. The $*$-subalgebra $C_c(G_0(\Per))\subseteq C^*(G_0(\Per))$ is closed under holomorphic functional calculus by \cite[Proposition 7.4]{DeeYas}, so by standard density and functional calculus arguments there exists a  projection $p\in C_c(G_0(\Per))$ and a unitary $u\in 1+C^*(G_0(\Per))$ such that $p=up_0u^*$. Since $p_0$ is full in $C^*(G_0(\Per))$, it follows that $p$ is full in $C^*(G_0(\Per))$. It is clear that $pC_c(G_0(\Per))p$ is closed under holomorphic functional calculus from \cite[Proposition 7.4]{DeeYas}.
\end{proof}

\begin{remark}
\label{expaneeded}
The authors were surprised by the length of the proof of Theorem \ref{existfull} for the following reason. The stable algebra $S=\varinjlim (C^*(G_0(\Per)),\alpha)$ is a simple $C^*$-algebra arising as an inductive limit, and by \cite{DGY} $S$ admits a full projection. The existence of projections in $S$ readily produces projections in $C^*(G_0(\Per))$, but fullness is not immediate even if it is conceivably provable from some soft argument. By the results of \cite{DeeGofYasFellAlgPaper}, there is no gain from $C^*(G_0(\Per))$ being a Fell algebra. As the following example shows, something extra is needed beyond the inductive limit, e.g., the proof we gave above relied on expansiveness. 

For a stationary inductive limit $S=\varinjlim (A,\alpha)$ where $\alpha$ is an injective, nondegenerate $*$-homomorphism, such that $S$ is simple and admits a projection (which is automatically full), one can ask if there is a full projection in $A$? The following example due to Jamie Gabe answers the question in the negative. Consider $A:=C_0(\N,\mathbb{K}(H))$ for an infinite-dimensional, separable Hilbert space $H$. Take a pair of Cuntz isometries $s$ and $t$ on $H$, i.e. $s$ and $t$ are isometries such that $ss^*+tt^*=1$. We define $\alpha:A\to A$ by 
$$\alpha(f)(n):=\begin{cases} sf(0)s^*+tf(1)t^*, \; &n=0,\\ f(n+1), \; &n>0.\end{cases}$$
It is clear from the construction that $\alpha$ is an injective, nondegenerate $*$-homomorphism. Moreover, consider two non-zero elements $f_1,f_2\in S$ in the image of any of the functorial embeddings $C_c(\N,\mathbb{K}(H))\hookrightarrow S$. Then for sufficiently large $N$, $\alpha^N(f_1),\alpha^N(f_2)\in S$ are both embeddings of non-zero elements in $A$ supported in $\{0\}\subseteq \N$. Since $\mathbb{K}(H)$ is simple, we have that $\overline{Sf_1S}=\overline{Sf_2S}$. From this observation, one can show that $S$ is simple.  In fact, one can show that $S\cong \mathbb{K}(H)$. Since $A$ admits no full projection, we can conclude our negative answer to the above question. One can note that in contrast to the expansive scenario of Theorem \ref{existfull}, the $*$-homomorphism in this counterexample is contracting on the spectrum of $A$.
\end{remark}

\section{Cuntz-Pimsner models for stationary inductive limits}
\label{cpsectionone}

In this section we consider a purely $C^*$-algebraic set up. However, the prototypical example is the case of the $C^*$-algebras associated to a Wieler solenoid. The general set up is as follows. We consider a non-degenerate self-$*$-monomorphism $\alpha:A\to A$ of a $C^*$-algebra $A$. The prototypical example is $A=C^*(G_0(\Per))$ with $\alpha$ given by the composition of $C^*(G_0(\Per)) \subseteq C^*(G_1(\Per))$ and the canonical isomorphism $C^*(G_1(\Per)) \cong C^*(G_0(\Per))$ from Theorem \ref{MainDeeYas}.

Before proceeding to Cuntz-Pimsner models, we make a remark about the stationary inductive limits.

\begin{prop}
Write $S:=\varinjlim (A,\alpha)$ and let $\phi:S\to S$ denote the right shift in the direct limit. Then $\phi$ is a well defined $*$-automorphism of $S$.
\end{prop}

\begin{proof}
The proof of this proposition is trivial once having parsed its statement. So parsing the statement we shall. By construction, $S:=\bigoplus_{k\in \N}A/I$ where $I$  is the ideal generated by $a\delta_k-\alpha(a)\delta_{k+1}$ for $a\in A$. Here we write $a\delta_k\in \bigoplus_{k\in \N}A$ for the element $a$ placed in position $k$. In this notation $\phi(a\delta_k\mod I)= a\delta_{k+1}\mod I$. The map $\phi$ is well defined since it is induced from the right shift mapping on $\bigoplus_{k\in \N}A$ that preserves $I$. The map $\phi$ is an inverse to the left shift mapping that coincides with the mapping $\alpha(a\delta_k+I):=\alpha(a)\delta_k+I$. Thus $\phi$ is a well defined $*$-automorphism.
\end{proof}

Associated with the $*$-monomorphism $\alpha:A\to A$ there is an $A$-Hilbert bimodule $E:={}_\alpha A_{\rm id}$. That is $E=A$ is a vector space with the action $a.\xi.b:=\alpha(a)\xi b$ for $a,b\in A$ and $\xi\in E$. The right $A$-inner product is given by 
$$\langle \xi_1,\xi_2\rangle:=\xi_1^*\xi_2\in A.$$
It is readily verified that $E$ is a right $A$-Hilbert module and that $A$ acts as adjointable operators on the left. Consider the associated Cuntz-Pimsner algebra $O_E$ and its core $\mathcal{C}_E:=O_E^{U(1)}$ for the standard gauge action.

\begin{prop}
\label{nonunitalcp}
The functorial property of inductive limits and the inclusion $A\hookrightarrow \mathcal{C}_E$ induces isomorphisms $S\cong \mathcal{C}_E$ and $S\rtimes_\phi \Z\cong O_E$.
\end{prop}

\begin{proof}
By \cite{MR1426840} it holds that 
$$\mathcal{C}_E=\varinjlim \mathbb{K}_A(E^{\otimes k}).$$
Since $\mathbb{K}_A(E^{\otimes k})=A$ and $\mathbb{K}_A(E^{\otimes k})\to \mathbb{K}_A(E^{\otimes k+1})$ coincides with $\alpha$, it follows that $S\cong \mathcal{C}_E$. The construction of $S\cong \mathcal{C}_E$ is by definition the map functorially constructed from the inclusion $A\hookrightarrow \mathcal{C}_E$. 

Following the standard procedure of ``extension of scalars" (see \cite{MR1426840} or \cite[Section 3.1]{MR4031052}) let $E_\infty:=E\otimes_A \mathcal{C}_E$ denote the self-Morita equivalence of $\mathcal{C}_E$ induced by $E$. For the details of the left action of $\mathcal{C}_E$ on $E_\infty$, see \cite[Proposition 3.2]{MR4031052}. There are natural isomorphisms 
$$O_E\cong O_{E_\infty}\cong \mathcal{C}_E\ltimes E_\infty.$$
For details, see \cite[Proposition 3.2]{MR4031052}. Here $\mathcal{C}_E\ltimes E_\infty$ denotes the generalized crossed product by the the self-Morita equivalence $E_\infty$. In terms of the isomorphism $S\cong \mathcal{C}_E$ we can identify the $\mathcal{C}_E$-self Morita equivalence $E_\infty$ with the $S$-self Morita equivalence ${}_\phi S_{\rm id}$. Therefore 
$$\mathcal{C}_E\ltimes E_\infty\cong S\rtimes {}_\phi S_{\rm id}=S\rtimes_\phi \Z.$$
\end{proof}

Proposition \ref{nonunitalcp} has implications for the stable Ruelle algebra of a Wieler solenoid. It was proven in \cite[Theorem 1.2]{DeeYas} that  $C^*(G^s(\Per))\cong \varinjlim(C^*(G_0(\Per)),\alpha)$ for the Fell subalgebra $C^*(G_0(\Per))\subseteq C^*(G^s(\Per))$ with spectrum $X^u(\Per)/{\sim_0}$. The next result extends this result to the stable Ruelle algebra, and is an immediate consequence of Proposition \ref{nonunitalcp} and \cite[Theorem 1.2]{DeeYas}.

\begin{corollary}
\label{wielernonunitalcp}
Consider a Wieler solenoid, and write $E$ for the $C^*(G_0(\Per)-C^*(G_0(\Per))$-correspondence defined as $E:=C^*(G_0(\Per))$ as a right Hilbert module with left action defined from $\alpha$. The stable Ruelle algebra $C^*(G^s(\Per))\rtimes \Z$ is isomorphic to the Cuntz-Pimsner algebra $O_E$. 
\end{corollary}

We are also interested in unital Cuntz-Pimsner models. We first describe the general setup and return to the specifics of Wieler solenoids at the end of the section. They can be constructed from choosing a full projection $p\in A$. For such a projection, define the unital $C^*$-algebra $A_p:=pAp$. Set $E_{p,k}:=\alpha^k(p)Ap$ equipped with the structure of an $A_p$-Hilbert bimodule defined by $a.\xi.b:=\alpha^k(a)\xi b$ for $a,b\in A$ and $\xi\in E$. The right $A_p$-inner product on $E_{p,k}$ is given by 
$$\langle \xi_1,\xi_2\rangle:=\xi_1^*\xi_2\in A_p.$$
We set $E_p:=E_{p,1}$. 

\begin{prop}
Let $p\in A$ be a full projection and define the unital $C^*$-algebra $A_p:=pAp$. It then holds that 
\begin{enumerate}
\item $E_{p,k}=E_p^{\otimes_A k}$ 
\item $\mathbb{K}_{A_p}(E_p^{\otimes_A k})=A_{\alpha^k(p)}$
\item $\mathcal{C}_{E_p}=\varinjlim (A_{\alpha^k(p)},\alpha)$ is a unital $C^*$-algebra which is stably isomorphic to $S$
\item $O_{E_p}$ is a unital $C^*$-algebra which is $U(1)$-equivariantly stably isomorphic to $S\rtimes_\phi \Z$.
\end{enumerate}
\end{prop}

\begin{proof}
For (1), we have for any $k,l\in \N$ that 
$$E_{p,k}\otimes_A E_{p,l}=\overline{\alpha^l(\alpha^k(p)Ap)\alpha^l(p)Ap}=\alpha^{k+l}(p)Ap=E_{p,k+l}.$$

For (2), we have that 
$$\mathbb{K}_{A_p}(E_p^{\otimes_A k})=E_p^{\otimes_A k}\otimes_A (E_p^{\otimes_A k})^*=\alpha^k(p)A\alpha^k(p)=A_{\alpha^k(p)}.$$

Item (3) follows from (2) and the fact that $\varinjlim (A_{\alpha^k(p)},\alpha)$ coincides with $pSp$ which is Morita equivalent to $S$. Item (4) follows from (3) and Proposition \ref{nonunitalcp}.
\end{proof}

We also note the following consequence on traces and KMS-weights that relies on \cite{MR2056837}.

\begin{thm}
\label{kmsything}
Let $(A,\alpha)$ be a stationary inductive system, $p\in A$ a full projection and $\beta\in \R$. There are bijections between the sets of following objects:
\begin{enumerate}
\item Tracial weights $\tau$ on $A$ such that $\tau\circ \alpha=\mathrm{e}^\beta \tau$ normalized by $\tau(p)=1$.
\item Tracial states $\tau_p$ on $A_p$ such that $\mathrm{Tr}_{\tau_p}^{E_p}=\mathrm{e}^\beta \tau_p$.
\item KMS${}_\beta$-weights $\Phi$ on $S\rtimes_\phi \Z$ normalized by $\Phi(p)=1$.
\end{enumerate}
The bijection from the set of objects in (1) to the set of objects in (2) is given by restriction along $A_p\hookrightarrow A$. The bijection from the set of objects in (3) to the set of objects in (1) is given by restriction along  $A\hookrightarrow S\rtimes_\phi \Z$.
\end{thm}

It is worth noting that the results from the previous section on the existence of full projections in the case of the Fell algebra associated to a Wieler solenoid imply that the previous theorem can be applied to the case of $(C^*(G_0(\Per)), \alpha)$ where $\alpha$ is the composition of $C^*(G_0(\Per)) \subseteq C^*(G_1(\Per))$ with the isomorphism $C^*(G_1(\Per)) \cong C^*(G_0(\Per))$. In the case of Wieler solenoids, we can summarize the results of this section in the following corollary. Recall that Corollary \ref{existefullforcom} ensures the existence of a projection $p\in C_c(G_0(\Per))$ which is full in $C^*(G_0(\Per))$.

\begin{corollary}
Let $(Y,\mathrm{d}_Y,g)$ be a Wieler pre-solenoid with associated Wieler solenoid $X:=\varprojlim (Y,g)$. We fix a projection $p\in C_c(G_0(\Per))$ which is full in $C^*(G_0(\Per))$. Consider the unital Fell algebra $A_p:=pC^*(G_0(\Per))p$ and define the $A_p$-bimodule $E_p:=\alpha(p)C^*(G_0(\Per))p$. It then holds that 
\begin{enumerate}
\item The Cuntz-Pimsner algebra $O_{E_p}$ is a unital $C^*$-algebra which is $U(1)$-equivariantly stably isomorphic to the stable Ruelle algebra $C^*(G^s(\Per))\rtimes \Z$
\item The core of the Cuntz-Pimsner algebra $\mathcal{C}_{E_p}=(O_{E_p})^{U(1)}$ is a unital $C^*$-algebra which is stably isomorphic to the stable algebra $C^*(G^s(\Per))$.
\item If $(Y,\mathrm{d}_Y,g)$ is mixing, the equation of traces on $A_p$ 
$$\mathrm{Tr}_{\tau}^{E_p}=\mathrm{e}^\beta \tau$$ 
only has a solution when $\beta$ is the topological entropy of $X$, in which case there exists a one-dimensional space of solutions determined from the Bowen measure on $X$ and the correspondence in Theorem \ref{kmsything} combined with item 1).
\end{enumerate}
\end{corollary}

For more details on the Bowen measure and how it determines traces and weights on the algebras associated with a Smale space, see \cite{DGY,MR3272757,Put}.

\section{The $K$-theory as a functor for compact, locally Hausdorff spaces} 
\label{SecKthFun}

The goal of this section is to extend the $K$-theory functor from compact, Hausdorff spaces to compact, locally Hausdorff spaces. As mentioned in Remark \ref{remarkonfellchar}, the process of taking spectra of Fell algebras does not define a functor when the morphisms between compact, locally Hausdorff spaces are taken to be continuous maps. Nevertheless, the $K$-theory of the associated Fell algebra is better behaved than the algebra itself. 

In addition to the contravariant functor associated to (at least some) continuous maps between compact, locally Hausdorff space, there is also a wrong-way functor for self local homeomorphisms. Wrong way functoriality is particularly important for our study of the Fell algebra associated to a Wieler solenoid because it appears both in the inductive limits structure of the stable algebra and the Cuntz--Pimsner model of the stable Ruelle algebra. 

We start by defining $K$-theory of a compact, locally Hausdorff space $\tilde{Y}$. As in Subsection \ref{fellalgsubsec}, we choose a Hausdorff resolution $\psi:X\to \tilde{Y}$. The associated groupoid $R(\psi)$ is defined in Example \ref{fellex}. We define 
$$K^*(\tilde{Y}):=K_*(C^*(R(\psi))).$$
This definition is a priori depending on the choice of $\psi$. We shall in the next subsection study functoriality and in the subsequent subsection show that up to canonical isomorphism, $K^*(\tilde{Y})$ is independent on the choice of $\psi$.

\subsection{Correspondences and maps between locally Hausdorff spaces}

We will in this subsection study how maps between compact, locally Hausdorff spaces give rise to correspondences between the related Fell algebras $C^*(R(\psi))$. Recall the terminology of a Hausdorff resolution of a compact, locally Hausdorff space $\tilde{Y}$ from Definition \ref{hausres}. We shall make use of another terminology.

\begin{definition}
Let $p_1:X_1\to \tilde{Y}_1$ and $p_2:X_2\to \tilde{Y}_2$ be Hausdorff resolutions of two compact, locally Hausdorff spaces. A proper morphism $(\Pi,\pi):p_1\to p_2$ is a pair of continuous maps fitting into the commuting diagram
 \begin{center}

\begin{tikzpicture}[node distance=1.5cm, scale=1, transform shape]
\node (X1)  {$X_1$};
\node (Y1) [below of = X1]{$\tilde{Y}_1$};
\node (Y2) [right of=Y1]{$\tilde{Y}_2$};
\node (X2) [right of=X1]{$X_2$};
\draw[->] (X1) to node[scale=0.8, above] {$\Pi$} (X2);
\draw[->] (X1) to node[scale=0.8, left] {$p_1$} (Y1);
\draw[->] (X2) to node[scale=0.8, left] {$p_2$} (Y2);
\draw[->] (Y1) to node [scale=0.8, above] {$\pi$} (Y2);
\end{tikzpicture}
\end{center}
such that $\Pi:X_1\to X_2$ is a proper mapping. We say that $\pi$ lifts to a proper morphism if there is a proper mapping $\Pi:X_1\to X_2$ making $(\Pi,\pi)$ into a proper morphism.
\end{definition}

The results of this subsection can be summarized in the following theorem.

\begin{thm}
Let $\tilde{Y}_1$ and $\tilde{Y}_2$ be compact, locally Hausdorff spaces, and assume that
$$\pi:\tilde{Y}_1\to \tilde{Y}_2,$$
is a continuous mapping. Fix Hausdorff resolutions $p_1:X_1\to \tilde{Y}_1$ and $p_2:X_2\to \tilde{Y}_2$. 
\begin{enumerate}
\item Associated with this data there is a canonically associated $C^*(R(p_2))-C^*(R(p_1))$-correspondence $\cor(\pi,p_1,p_2)$, see Definition \ref{correfrompi}.
\item The left action of $C^*(R(p_2))$ on $\cor(\pi,p_1,p_2)$ is via $C^*(R(p_1))$-compact operators if $\pi$ lifts to a proper morphism for some Hausdorff resolutions of $\tilde{Y}_1$ and $\tilde{Y}_2$, for instance if $\pi$ is a local homeomorphism. In fact, if $\pi$ is a homeomorphism then $\cor(\pi,p_1,p_2)$ is a Morita equivalence.
\item If $\pi':\tilde{Y}_2\to \tilde{Y}_3$ is another continuous mapping to a space with a Hausdorff resolution $p_3:X_3\to \tilde{Y}_3$, then there is a unitary isomorphism of $C^*(R(p_3))-C^*(R(p_1))$-correspondences
$$\cor(\pi',p_2,p_3)\otimes_{C^*(R(p_2))}\cor(\pi,p_1,p_2)\cong \cor(\pi'\circ \pi,p_1,p_3).$$
\end{enumerate}

\end{thm}

We shall start with an easy result for proper morphisms of Hausdorff resolutions.

\begin{prop}
Given a proper morphism $(\Pi,\pi):p_1\to p_2$ of Hausdorff resolutions $p_1:X_1\to \tilde{Y}_1$ and $p_2:X_2\to \tilde{Y}_2$, the pullback map 
$$\Pi^*:C_c(R(p_2))\to C_c(R(p_1)),$$
along the proper groupoid homomorphism $(x,x')\mapsto(\Pi(x),\Pi(x'))$, is a well defined $*$-homomorphism that induces a $*$-homomorphism
$$\Pi^*:C^*(R(p_2))\to C^*(R(p_1)).$$
\end{prop}

Recall the notion of a groupoid correspondence from Definition \ref{groupoidcorr}.

\begin{definition}
\label{def} 
For Hausdorff resolutions $p_1:X_1\to \tilde{Y}_1$ and $p_2:X_2\to \tilde{Y}_2$, we consider a diagram $\mathfrak{X}$ of the form
 \begin{center}

\begin{tikzpicture}[node distance=1.5cm, scale=1, transform shape]
\node (X1)  {$X_1$};
\node (Y1) [below of = X1]{$\tilde{Y}_1$};
\node (Y2) [right of=Y1]{$\tilde{Y}_2$};
\node (X2) [right of=X1]{$X_2$};
\draw[->] (X1) to node[scale=0.8, left] {$p_1$} (Y1);
\draw[->] (X2) to node[scale=0.8, left] {$p_2$} (Y2);
\draw[->] (Y1) to node [scale=0.8, above] {$\pi$} (Y2);
\end{tikzpicture}
\end{center}
where $\pi$ is a continuous map. From the diagram $\mathfrak{X}$, we define the subspace $Z_\mathfrak{X}\subseteq X_1\times X_2$ by 
\[ Z_\mathfrak{X}= \{ (x,y)\in X_1\times X_2: \pi\circ p_1(x) = p_2(y)\}.\]
\end{definition}

\begin{lemma}
\label{corres}
The space $Z_\mathfrak{X}$ defined above is a groupoid correspondence from $R(p_1)$ to $R(p_2)$.
\end{lemma}

\begin{proof} 
The moment map for the right action will be defined as $\sigma: Z_\mathfrak{X}\rightarrow G_2^{(0)}$, $(x,y)\mapsto (y,y)$. The map $\sigma$ is an open map. Define the right action of $R(p_2)$ on $Z_\mathfrak{X}$ by $(x,y)\cdot (y,y')=(x,y')$ for $(x,y)\in Z_\mathfrak{X}$ and $(y,y')\in R(p_2)$. This action is free and proper. 

Define the moment map $\rho: Z_\mathfrak{X}\rightarrow R(p_1)^{(0)}$ for the left action by $(x,y)\mapsto (x,x)$, which is open and surjective. Define the action of $R(p_1)$ on $Z_\mathfrak{X}$  by    $(x_1,x_2)\cdot (x_2,y)=(x_1,y),$ for $(x_1,x_2)\in R(p_1), (x_2,y)\in Z_\mathfrak{X}$ which is free and proper. Moreover, $\rho$ induces a homemorphism $\overline{\rho}: Z_\mathfrak{X} / R(p_2) \rightarrow R(p_1)^{(0)}.$ 
\end{proof}

The following lemma follows from Theorem~\ref{corGroupoid}. 

\begin{prop}
\label{lem} 
Let $Z_\mathfrak{X}$ be the groupoid correspondence defined from a diagram $\mathfrak{X}$ as in Definition~\ref{def}. Then, the space $C_c(Z_\mathfrak{X})$ becomes a pre-correspondence from $C_c(R(p_2))$ to $C_c(R(p_1))$ with the operations given as in Theorem~\ref{corGroupoid}.
\end{prop} 

\begin{definition}
\label{correfrompi}
Given a diagram $\mathfrak{X}$ as in Definition~\ref{def}, we define the correspondence $\cor(\pi,p_1,p_2)$ from $C^*(R(p_2))$ to $C^*(R(p_1))$ as the completion of the pre-correspondence $C_c(Z_\mathfrak{X})$. In other words $\cor(\pi,p_1,p_2)$ is a $C^*(R(p_2))-C^*(R(p_1))$-Hilbert $C^*$-module. 
\end{definition}

\begin{lemma} 
\label{moritainvarcorr}
Let $Z_\mathfrak{X}$ be the groupoid correspondence defined from a diagram $\mathfrak{X}$ as in Definition~\ref{def}. If $\pi$ is a continuous bijection then $\cor(\pi,p,q)$ is an imprimitivity bimodule from $C^*(R(p_2))$ to $C^*(R(p_1))$ in the left $C^*(R(p_2))$-valued inner product defined as follows. The $C_c(R(p_2))$-valued inner product $\<\<\xi_1, \xi_2\>\>\in C_c(R(p_2))$ of $\xi_1, \xi_2 \in C_c(Z_\mathfrak{X})$ as
\begin{equation}
\label{left}
\<\<\xi_1, \xi_2\>\>(y,y') = \sum_{\substack{ x\in X with \\ (x,x')\in R(p_1)}} \xi_1(x,y')\overline{\xi_2(x,y)}, \quad (y,y')\in R(p_2),
\end{equation}
for $x'\in X$ with $(x',y)\in Z_\mathfrak{X}$, and this left inner product is extended to a left $C^*(R(p_2))$-valued inner product on $\cor(\pi,p_1,p_2)$ by continuity. 
\end{lemma}

\begin{proof} 
Notice that $\sigma: Z_\mathfrak{X}\rightarrow R(p_2)^{(0)}$ is an open map; and since $\pi$ is a bijection, the map $\sigma$ induces a bijection from $R(p_1)\backslash Z_\mathfrak{X}$ onto $R(p_1)^{(0)}$. Moreover,$Z_\mathfrak{X}$ is a left principal $R(p_1)$-space and a right principal $R(p_2)$-space. Now, as in \cite{MRW}, the operation (\ref{left}) defines a left semi-inner product, and the correspondence $\pre {C_c(R(p_2))}C_c(Z_\mathfrak{X})_{C_c(R(p_1))}$ becomes a pre-imprimitivity bimodule. 
\end{proof} 

\begin{example} 
Let $X_1,X_2$ be compact Hausdorff spaces and let $\pi: X_1\rightarrow X_2$ be a continuous map. The map $\pi$ induces a unital homomorphism via pulback
\[ \pi^*: C(X_2)\rightarrow C(X_1), \hspace{1cm} \pi^*(g)(x)=f\left(\pi(x)\right),\]
for $g\in C(X_2),$ $x\in X_1$. The homomorphism $\pi^*$ allows one to view $C(X_1)$ as a $C(X_2)-C(X_1)$ correspondence. We show that this correspondence is isomorphic to the correspondence $\pre {C(X_2)}C(Z_\mathfrak{X})_{C(X_1)}$ associated to the following diagram.
 \begin{center}

\begin{tikzpicture}[node distance=1.5cm, scale=1, transform shape]
\node (X1)  {$X_1$};
\node (W1) [below of = X1]{$X_1$};
\node (W2) [right of=W1]{$X_2$};
\node (X2) [right of=X1]{$X_2$};
\draw[->] (X1) to node[scale=0.8, left] {$\id_{X_1}$} (W1);
\draw[->] (X2) to node[scale=0.8, right] {$\id_{X_2}$} (W2);
\draw[->] (W1) to node [scale=0.8, above] {$\pi$} (W2);
\end{tikzpicture}
\end{center}

We have $Z_\mathfrak{X}= \{ (x,\pi(x))\in X_1\times X_2: x\in X_1\}$. Notice that $R(\id_{X_j})=\{(x,x):x\in X_j\}$ for $j=1,2$. Therefore, we may identify the groupoid $R(\id_{X_j})$ with the trivial groupoid $X_j$, for $j=1,2$. The diagram above gives us the correspondence $\pre {C(X_2)}C(Z_\mathfrak{X})_{C(X_1)}$ with the operations
\begin{align*}
\xi\cdot f \left(x,\pi(x)\right) &=\xi\left(x,\pi(x)\right)f(x)\\
g\cdot \xi(x,\pi(x)) &=g\left(\pi(x)\right)\xi\left(x,\pi(x)\right)\\
\<\xi_1,\xi_2\>(x) &= \overline{\xi_1\left(x,\pi(x)\right)}\xi_2\left(x,\pi(x)\right)
\end{align*} 
for $\xi, \xi_1,\xi_2\in C(Z_\mathfrak{X}), g\in C(X_2),$ and $f\in C(X_1)$. We can define the isomorphism of correspondences $\phi: C(Z_\mathfrak{X})\rightarrow C(X)$ by 
\[ \phi(\xi)(x)=\xi(x,\pi(x)).\]
Indeed, it is clear that $\phi$ has the inverse $\phi^{-1}(f)(x,\pi(x))=f(x)$ and a short computation shows that $\phi$ is compatible with the correspondence structure.

\end{example}

\begin{lemma}
\label{complemma}
Assume we have the following diagram
 \begin{center}
\begin{tikzpicture}[node distance=1.5cm, scale=1, transform shape]
\node (X1)  {$X_1$};
\node (W1) [below of = X1]{$\tilde{Y}_1$};
\node (W2) [right of=W1]{$\tilde{Y}_2$};
\node(W3)[right of=W2]{$\tilde{Y}_3$};
\node (X2) [right of=X1]{$X_2$};
\node (X3) [right of=X2]{$X_3$};
\draw[->] (X1) to node[scale=0.8, left] {$p_1$} (W1);
\draw[->] (X2) to node[scale=0.8, left] {$p_2$} (W2);
\draw[->] (X3) to node[scale=0.8, left] {$p_3$} (W3);
\draw[->] (W2) to node[scale=0.8, above] {$\pi_2$} (W3);
\draw[->] (W1) to node [scale=0.8, above] {$\pi_1$} (W2);
\end{tikzpicture}
\end{center}
where $p_1:X_1\to \tilde{Y}_1$, $p_2:X_2\to \tilde{Y}_2$ and $p_3:X_3\to \tilde{Y}_3$ are Hausdorff resolutions; and $\pi_1$ and $\pi_2$ are continuous maps. Write $\mathfrak{X}_1$ for the left part of the diagram and $\mathfrak{X}_2$ for the right part and $\mathfrak{X}_3$ for the diagram 
 \begin{center}
\begin{tikzpicture}[node distance=1.5cm, scale=1, transform shape]
\node (X1)  {$X_1$};
\node (Y1) [below of = X1]{$\tilde{Y}_1$};
\node (Y2) [right of=Y1]{$\tilde{Y}_3$};
\node (X2) [right of=X1]{$X_3$};
\draw[->] (X1) to node[scale=0.8, left] {$p_1$} (Y1);
\draw[->] (X2) to node[scale=0.8, left] {$p_3$} (Y2);
\draw[->] (Y1) to node [scale=0.8, above] {$\pi_2\circ \pi_1$} (Y2);
\end{tikzpicture}
\end{center}
For $f\in C_c(Z_{\mathfrak{X}_2})$, $\xi\in C_c(Z_{\mathfrak{X}_1})$, the map $F(f,\xi)$ defined by 
\[ F(f,\xi)(x,r)=\sum_{\substack{y\in Y with \\\pi\circ p_1(x)=p_2(y)}} \xi(x,y)f(y,r)\]
is an element of $C_c(Z_{\mathfrak{X}_3})$. 
\end{lemma}

\begin{proof} 
Consider the closed subset \[Z_{\mathfrak{X}_1} * Z_{\mathfrak{X}_2}=\left\{ \big((x,y),(y,r)\big): (x,y) \in  Z_{\mathfrak{X}_1} \hspace{.2cm} and \hspace{.2cm}   (y,r) \in Z_{\mathfrak{X}_2} \right\}\] of $Z_{\mathfrak{X}_1}\times Z_{\mathfrak{X}_2}$, and the continuous surjective map 
\[ V: Z_{\mathfrak{X}_1} * Z_{\mathfrak{X}_2} \rightarrow Z_{\mathfrak{X}_3}, \hspace{.5cm} \left((x,y),(y,r)\right)  \mapsto (x,r).\] 
Give $f$ and $\xi$ as in the statement of the lemma, we define the map $g: Z_{\mathfrak{X}_1}* Z_{\mathfrak{X}_2}\rightarrow \mathbb{C}$ by 
\[ g\left((x,y),(y,r)\right)=\xi(x,y)f(y,r).\]
Then $\mathrm{supp}(g)$ is contained in the compact $\mathrm{supp}(\xi)\times\mathrm{supp}(f)$, and thus $g$ is compactly supported. We now show that $\mathrm{supp}(F(f,\xi))\subset V(\mathrm{supp}g)$. Assume $(x,r)\notin V(\mathrm{supp}g)$. Then, we have $g(m)=0$ for any $m\in V^{-1}\left((x,r)\right).$  Since 
\[ F(f,\xi)(x,r)= \sum_{m\in V^{-1}\left((x,r)\right)}g(m),\]
we have $F(f,\xi)(x,r)=0,$ which completes the proof. 
\end{proof}

\begin{thm}
\label{comp} 
In the notation of Lemma \ref{complemma}, the map $\Phi$ defined by 
\[C_c(Z_{\mathfrak{X}_2}) \otimes C_c(Z_{\mathfrak{X}_1}) \rightarrow C_c(Z_{\mathfrak{X}_3}), \hspace{.5cm} \text{$f\otimes \xi \mapsto F(f,\xi)$}\]
is a  $C_c(R(p_3)) - C_c(R(p_1))$ pre-correspondence isomorphism. The map $\Phi$ induces an isomorphism of $C^*(R(p_3)) - C^*(R(p_1))$-correspondences
$$\cor(\pi_2,p_2,p_3)\otimes_{C^*(R(p_2))}\cor(\pi_1,p_1,p_2)\cong \cor(\pi_2\circ \pi_1,p_1,p_3).$$
\end{thm}

The proof is omitted as it only consists of long computations verifying that the module structure and inner products are respected by $\Phi$. The following immediate corollary of Theorem \ref{comp} shows that the correspondence $\cor(\pi,p_1,p_2)$ only depends on $\pi$ up to canonical Morita equivalence.

\begin{corollary}
Assume that $\pi:\tilde{Y}\to \tilde{Y}'$ is a continuous mapping between compact, locally Hausdorff spaces. Let $p_j: X_j\rightarrow \tilde{Y}$ and $p_j': X_j'\rightarrow \tilde{Y'}$ be surjective local homeomorphisms from locally compact, Hausdorff spaces, for $j=1,2$. Then, for $j=1,2$, $\cor(\pi, p_1,p_j')$ and $\cor(\pi, p_2, p_j')$ are Morita equivalent via $\cor(\id_{\tilde{Y}},p_1,p_2)$ and $\cor(\pi, p_j,p_1')$ and $\cor(\pi, p_j, p_2')$ are Morita equivalent via $\cor(\id_{\tilde{Y}'},p_1',p_2')$.
\end{corollary}

\begin{lemma}
\label{propcomplemma}
In the notation of Lemma \ref{complemma}, if $\pi_1$ lifts to a proper morphism $(\Pi_1,\pi_1):p_1\to p_2$ then 
$$\cor(\pi_2,p_2,p_3)\otimes_{\Pi_1^*}C^*(R(p_2))\cong \cor(\pi_2\circ \pi_1,p_1,p_3).$$
Similarly, if $\pi_2$ lifts to a proper morphism $(\Pi_2,\pi_2):p_2\to p_3$ then 
$$C^*(R(p_3))\otimes_{\Pi_2^*}\cor(\pi_1,p_1,p_2)\cong \cor(\pi_2\circ \pi_1,p_1,p_3).$$
In particular, if $(\Pi,\pi):p_1\to p_2$ is a proper morphism then there is a Morita equivalence of correspondences from $\cor(\pi,p_1,p_2)$ the $C^*(R(p_2))-C^*(R(p_1))$-correspondence ${}_{\Pi^*}C^*(R(p_1))$, i.e. the right $C^*(R(p_1))$-Hilbert module $C^*(R(p_1))$ with left action defined from $\Pi^*$.
\end{lemma}

\begin{proof}
The proof of the first two isomorphisms is analogous to Lemma \ref{complemma} and Theorem \ref{comp} and is omitted. To prove the final statement, we note that $\cor(\id_{\tilde{Y}_1},p_1,p_1)=C^*(R(p_1))$ and so
$${}_{\Pi^*}C^*(R(p_1))=C^*(R(p_2))\otimes_{\Pi^*}\cor(\id_{\tilde{Y}_1},p_1,p_1)\cong\cor(\pi,p_1,p_2).$$
\end{proof}

\begin{lemma}
\label{nicemaps}
Consider a diagram as in Definition~\ref{def}. The left action of $C^*(R(p_2))$ on $\cor(\pi,p_1,p_2)$ is via $C^*(R(p_1))$-compact operators if $\pi$ lifts to a proper morphism for some Hausdorff resolutions of $\tilde{Y}_1$ and $\tilde{Y}_2$.
\end{lemma}

\begin{proof}
The statement that the left action of $C^*(R(p_2))$ on $\cor(\pi,p_1,p_2)$ is via $C^*(R(p_1))$-compact operators is Morita invariant, so by Lemma \ref{moritainvarcorr} it suffices to prove the lemma for particular choices of Hausdorff resolutions $p_1$ and $p_2$. If it is the case that $\pi$ lifts to a proper morphism for some Hausdorff resolutions of $\tilde{Y}_1$ and $\tilde{Y}_2$, we can therefore assume that $\pi$ lifts to a proper morphism $p_1\to p_2$. In this case, the left action of $C^*(R(p_2))$ on $\cor(\pi,p_1,p_2)$ is via $C^*(R(p_1))$-compact operators by the final statement of Lemma \ref{propcomplemma}.
\end{proof}

\subsection{Functoriality in $K$-theory of compact, locally Hausdorff spaces}
\label{funcink}
In this section we consider compact, locally Hausdorff spaces. We are interested in their $K$-theory. First we show that the $K$-theory of compact, locally Hausdorff spaces is well-defined.

\begin{prop}
\label{geenknd}
Consider a compact, locally Hausdorff space $\tilde{Y}$. Defining the $K$-theory of $\tilde{Y}$ as
$$K^*(\tilde{Y}):=K_*(C^*(R(\psi))),$$
for some Hausdorff resolution $\psi:X\to\tilde{Y}$, produces a group uniquely determined up to canonical isomorphism.
\end{prop}

\begin{proof}
If $\psi_1$ and $\psi_2$ are two different Hausdorff resolutions of $\tilde{Y}$, Lemma \ref{moritainvarcorr} shows that $\cor(\id_{\tilde{Y}},\psi_1,\psi_2)$ is a Morita equivalence producing the sought after isomorphism $K_*(C^*(R(\psi_1)))\cong K_*(C^*(R(\psi_2)))$.
\end{proof}

Next, we study the contravariant properties of $K$-theory for a sub-class of continuous mappings. In compliance with the results of the last subsection, we say that a continuous map 
$$\pi:\tilde{Y}_1\to \tilde{Y}_2,$$
is HRP (Hausdorff Resolution Proper) if $\pi$ lifts to a proper morphism $p_1\to p_2$ for some Hausdorff resolutions $p_1:X_1\to \tilde{Y}_1$ and $p_2:X_2\to \tilde{Y}_2$. We note the following consequence of Lemma \ref{propcomplemma} and \ref{nicemaps}. 

\begin{thm}
If the continuous map of compact, locally Hausdorff spaces 
$$\pi:\tilde{Y}_1\to \tilde{Y}_2,$$
is HRP, there is an associated class $[\pi]\in KK_0(C^*(R(p_2)),C^*(R(p_1)))$ for any Hausdorff resolutions $p_1:X_1\to \tilde{Y}_1$ and $p_2:X_2\to \tilde{Y}_2$. Moreover, if $\pi':\tilde{Y}_2\to \tilde{Y}_3$ is another HRP-map, then we have the following Kasparov product
$$[\pi']\otimes_{C^*(R(p_2))}[\pi]=[\pi'\circ \pi]\in KK_0(C^*(R(p_3)),C^*(R(p_1))),$$
for any Hausdorff resolution $p_3:X_3\to \tilde{Y}_3$.
\end{thm}

Contravariance of $K$-theory of compact, locally Hausdorff spaces under HRP-maps is now immediate.

\begin{corollary}
Taking $K$-theory of compact, locally Hausdorff spaces defines a contravariant functor from the category of compact, locally Hausdorff spaces with morphisms being HRP-maps to the category of $\Z/2$-graded abelian groups. 
\end{corollary}

We now turn to wrong way functoriality. Assume that $\tilde{g}: \tilde{Y}_1 \rightarrow \tilde{Y}_2$ is a surjective local homeomorphism. Our goal is the definition of a wrong way map from the $K$-theory of $\tilde{Y}_1$ to that of $\tilde{Y}_2$. Take a Hausdorff resolution $p: X \rightarrow \tilde{Y}_1$. The following diagram is clearly commutative:
\[
\begin{CD}
X@>\id_X >> X \\
@V p VV @VV \tilde{g}\circ p V   \\
\tilde{Y}_1 @>\tilde{g} >> \tilde{Y}_2  \\
\end{CD}.
\]
Since the vertical maps are both surjective local homeomorphisms, $\tilde{g}\circ p: X \rightarrow \tilde{Y}_2$ is a Hausdorff resolution. We can form the \'etale groupoids $R(p)$ and $R(\tilde{g}\circ p)$ over $X$, whose assocated groupoid algebras have spectrum $\tilde{Y}_1$ and $\tilde{Y}_2$, respectively. Furthermore, $R(p) \subseteq R(\tilde{g}\circ p)$ as an open subgroupoid. This inclusion at the groupoid level leads to an inclusion of $C^*$-algebras, $C^*(R(p)) \subseteq C^*(R(\tilde{g} \circ p))$. The induced map on  $K$-theory will be denoted by 
\[ \tilde{g}!: K^*(\tilde{Y}_1)=K_*(C^*(R(p))) \rightarrow K_*(C^*(R(\tilde{g} \circ p)))=K^*(\tilde{Y}_2). \]
Following the same argument as in the proof of Proposition \ref{geenknd}, we arrive at the following.

\begin{prop}
If $\tilde{g}: \tilde{Y}_1 \rightarrow \tilde{Y}_2$ is a surjective local homeomorphism of compact, locally Hausdorff spaces, then the map 
\[ \tilde{g}!: K^*(\tilde{Y}_1)\rightarrow K^*(\tilde{Y}_2). \] 
is well-defined, i.e. independent of Hausdorff resolution.
\end{prop}

The next proposition shows that we indeed have wrong way functoriality for surjective local homeomorphisms of compact, locally Hausdorff spaces.

\begin{prop}
\label{wrongwayfuncfunc}
If $\tilde{g}_1: \tilde{Y}_1 \rightarrow \tilde{Y}_2$ and $\tilde{g}_2: \tilde{Y}_2 \rightarrow \tilde{Y}_3$ are surjective local homeomorphisms of compact, locally Hausdorff spaces, then 
\[ \tilde{g}_2!\circ \tilde{g}_1!=(\tilde{g}_2\circ \tilde{g}_1)!: K^*(\tilde{Y}_1)\rightarrow K^*(\tilde{Y}_3). \] 
\end{prop}

\begin{proof}
The claim follows from that both sides can be defined from the chain of inclusions 
$$C^*(R(p)) \subseteq C^*(R(\tilde{g}_1 \circ p))\subseteq C^*(R(\tilde{g}_2 \circ \tilde{g}_1 \circ p))$$
where $p:X\to \tilde{Y}_1$ is a Hausdorff resolution. The first inclusion defines $\tilde{g}_1!$, the second inclusion defines $\tilde{g}_2!$ and the composed inclusion defines $(\tilde{g}_2\circ \tilde{g}_1)!$.
\end{proof}

\begin{remark}
We note that also for wrong way maps, we could also work at the level of $KK$-theory. Associated with a surjective local homeomorphism $\tilde{g}: \tilde{Y}_1 \rightarrow \tilde{Y}_2$, we can for any Hausdorff resolutions $p_1:X_1\to \tilde{Y}_1$ and $p_2:X_2\to \tilde{Y}_2$ define a wrong way class 
$$[g!]:=\iota^*\cor(\id_{\tilde{Y}_2},p_2,\tilde{g}\circ p_1)\in KK_0(C^*(R(p_1)),C^*(R(p_2))),$$
where $\iota$ denotes the ideal inclusion $C^*(R(p_1)) \subseteq C^*(R(\tilde{g} \circ p_1))$. We can then identify $g!$ with the Kasparov product from the right 
$$\cdot\otimes_{C^*(R(p_1))} [g!]:KK_*(\C,C^*(R(p_1)))\to KK_*(\C,C^*(R(p_2))),$$
under the natural isomorphisms 
$$KK_*(\C,C^*(R(p_1)))\cong K_*(C^*(R(p_1)))\quad\mbox{and}\quad KK_*(\C,C^*(R(p_2)))\cong K_*(C^*(R(p_2))).$$ 
In this setting, the analogue of Proposition \ref{wrongwayfuncfunc} is the Kasparov product statement that
$$[g_1!]\otimes_{C^*(R(p_2))}[g_2!]=[(g_2\circ g_1)!]\in KK_0(C^*(R(p_1)),C^*(R(p_3))),$$
which can be verified using Theorem \ref{comp} and Lemma \ref{propcomplemma}. 
\end{remark}

\begin{example}
\label{wrongforwiel}
If we consider the non-Hausdorff dynamical system $(X^u(\Per)/{\sim_0},\tilde{g})$ associated with a Wieler solenoid, and the Hausdorff resolution $q:X^u(\Per)\to X^u(\Per)/{\sim_0}$, then as an element of $KK_0(C^*(G_0(\Per)),C^*(G_0(\Per)))$ we have that 
$$[g!]:=\iota^*\cor(\id_{X^u(\Per)/{\sim_0}},q,\tilde{g}\circ q)=[E],$$
where $E$ is the $C^*(G_0(\Per))-C^*(G_0(\Per))$-correspondence defined as $E:=C^*(G_0(\Per))$ as a right Hilbert module with left action defined from $\alpha$. The correspondence $E$ was considered in Corollary \ref{nonunitalcp} and will play a role in $K$-theory computations below.
\end{example}

\begin{example}
\label{localhomeoexample}
Let $g:Y_1\to Y_2$ be a surjective local homeomorphism of compact, Hausdorff spaces. Write $E_g$ for the finitely generated projective $C(Y_2)$-module $C(Y_1)$, with right action defined from $g^*$. We can equipp $E_g$ with the $C(Y_2)$-valued right inner product 
$$\langle f_1,f_2\rangle(x):=\sum_{y\in g^{-1}(x)} \overline{f_1(y)}f_2(y), \quad x\in Y_2, \; f_1,f_2\in E_g=C(Y_1).$$
The algebra $C(Y_1)$ acts adjointably from the left on $E_g=C(Y_1)$ by pointwise multiplication. With these structures, $E_g$ is a $C(Y_1)-C(Y_2)$-correspondence in which $C(Y_1)$ acts as $C(Y_2)$-compact operators. 

It is immediate from the definition of the inner product on $E_g$ that $\mathbb{K}_{C(Y_2)}(E_g)=C^*(R(g))$. By taking $p=\id_{Y_1}$ in the constructions above, it follows that 
$$[g!]=[E_g]\in KK_0(C(Y_1),C(Y_2)).$$
In the case that $Y=Y_1=Y_2$, the class $[E_g]\in KK_0(C(Y),C(Y))$ plays an important role in the $K$-theory, or more generally $KK$-theory, of the Cuntz-Pimsner algebra associated with the module $E_g$ \cite{GMR,MR1426840}. For Wieler solenoids this is of interest due to the results of the next subsection.
\end{example}

\subsection{Wrong way maps and $K$-theory of the stable and unstable Ruelle algebra}

In Corollary \ref{nonunitalcp} we saw that the stable Ruelle algebra $C^*(G^s(\Per))\rtimes \Z$ of a Wieler solenoid is a Cuntz-Pimsner algebra over $C^*(G_0(\Per))$. From the general theory of Cuntz-Pimsner algebras, the $K$-theory of the stable Ruelle algebra can therefore be computed from $K^*(X^u(\Per)/{\sim_0})$ -- the $K$-theory of $C^*(G_0(\Per))$. Using Kaminker-Putnam-Whittaker duality result \cite{KPW}, the $K$-theory of the unstable Ruelle algebra can similarly be computed from the $K$-homology group $K_*(X^u(\Per)/{\sim_0}):=K^*(C^*(G_0(\Per)))$. Throughout the subsection, we tacitly identify the stable Ruelle algebra $C^*(G^s(\Per))\rtimes \Z$ with the Cuntz-Pimsner algebra $O_E$ over $C^*(G_0(\Per))$ in Corollary \ref{nonunitalcp}.

\begin{thm}
Assume that $(X,\phi)$ is a Wieler solenoid and write $(X^u(\Per)/{\sim_0},\tilde{g})$ for the associated non-Hausdorff dynamical system. The $K$-theory of the stable Ruelle algebra $K_*(C^*(G^s(\Per))\rtimes \Z)$ fits into a six term exact sequence with the $K$-theory of $X^u(\Per)/{\sim_0}$:
$$\begin{CD}
K^0(X^u(\Per)/{\sim_0}) @>1-\tilde{g}!>> K^0(X^u(\Per)/{\sim_0}) @>j_S>>  K_0(C^*(G^s(\Per))\rtimes \Z)\\
@A\beta_S AA @. @VV\beta_S V \\
 K_1(C^*(G^s(\Per))\rtimes \Z) @<j_S<<  K^1(X^u(\Per)/{\sim_0}) @<1-\tilde{g}!<< K^1(X^u(\Per)/{\sim_0})
\end{CD}$$
where $j_S:C^*(G_0(\Per))\hookrightarrow C^*(G^s(\Per))\rtimes \Z$ denotes the inclusion and $\beta_S$ is the Pimsner boundary map in $KK_1(C^*(G^s(\Per))\rtimes \Z,C^*(G_0(\Per)))$, cf. \cite{MR4031052,GMR,MR1426840}.
\end{thm}

\begin{proof}
By standard results for Cuntz-Pimsner algebras \cite{MR4031052,GMR,MR1426840}, we have a six term exact sequence 
$$\begin{CD}
K_0(C^*(G_0(\Per))) @>1-\cdot\otimes_{C^*(G_0(\Per))}[E]>> K_0(C^*(G_0(\Per))) @>j_S>>  K_0(C^*(G^s(\Per))\rtimes \Z)\\
@A\beta_S AA @. @VV\beta_S V \\
 K_1(C^*(G^s(\Per))\rtimes \Z) @<j_S<<  K_1(C^*(G_0(\Per))) @<1-\cdot\otimes_{C^*(G_0(\Per))}[E]<< K_1(C^*(G_0(\Per)))
\end{CD}$$
where $E$ is the $C^*(G_0(\Per))-C^*(G_0(\Per))$-correspondence from Corollary \ref{nonunitalcp}. The theorem now follows from the identity $[\tilde{g}!]=[E]\in KK_0(C^*(G_0(\Per)),C^*(G_0(\Per)))$ in Example \ref{wrongforwiel}.
\end{proof}

Using Kaminker-Putnam-Whittaker duality result \cite{KPW}, we can derive an analogous result for the $K$-theory of the unstable Ruelle algebra $C^*(G^u(\Per))\rtimes \Z$. Here one uses $K$-homology of $X^u(\Per)$ instead of $K$-theory, and the theory for $K$-homology of compact, locally Hausdorff spaces can be developed in the same way as that of $K$-theory in Subsection \ref{funcink}. 

\begin{thm}
Assume that $(X,\phi)$ is a Wieler solenoid and write $(X^u(\Per)/{\sim_0},\tilde{g})$ for the associated non-Hausdorff dynamical system. 
The $K$-theory of the unstable Ruelle algebra $K_*(C^*(G^u(\Per))\rtimes \Z)$ fits into a six term exact sequence the $K$-homology of $X^u(\Per)/{\sim_0}$:
$$\begin{CD}
K_0(X^u(\Per)/{\sim_0}) @>1-[\tilde{g}!]\otimes\cdot>> K_0(X^u(\Per)/{\sim_0}) @>\beta_U>>  K_0(C^*(G^u(\Per))\rtimes \Z)\\
@Aj_UAA @. @VVj_U V \\
 K_1(C^*(G^u(\Per))\rtimes \Z) @<\beta_U<<  K_1(X^u(\Per)/{\sim_0}) @<1-[\tilde{g}!]\otimes\cdot<< K_1(X^u(\Per)/{\sim_0})
\end{CD}$$
where $j_U$ is Kaminker-Putnam-Whittaker dual to $j_S$ and $\beta_U$ is Kaminker-Putnam-Whittaker dual to the Pimsner boundary map in $KK_1(C^*(G^s(\Per))\rtimes \Z,C^*(G_0(\Per)))$, cf. \cite{MR4031052,GMR,MR1426840}.
\end{thm}

\subsection{Groupoid homology}
Although we have only discussed the $K$-theory of stable and stable Ruelle groupoid $C^*$-algebras, one can also use the results in this paper to compute the homology of these groupoids. For the stable groupoid, one uses the main result of \cite{DeeYas} (stated above in Theorem \ref{MainDeeYas}) and \cite[Proposition 4.7 ]{MR4030921}. Again, determining the map (now on groupoid homology) associated to the open inclusion $G_0(\Per) \subseteq G_1(\Per)$ is the key to explicit computations. For the stable Ruelle groupoid, one uses the homology of the stable groupoid and \cite[Lemma 1.3]{MR4170644} (also see \cite[Theorem 3.8]{MR4283513}).

\subsection{Examples}
In this section of the paper we discuss a few explicit examples. In the first few, the original map $g$ is a local homeomorphism and hence $X^u(\Per)/{\sim_0}=Y$ and $\tilde{g}=g$. Nevertheless, these examples illustrate the importance of computing the wrong-way map in $K$-theory computation. We also summarize the case of the aab/ab-solenoid where $X^u(\Per)/{\sim_0} \neq Y$ and $\tilde{g} \neq g$, which is discussed in much more detail in \cite{DeeYas}. In fact, typically $g: Y \rightarrow Y$ is not a local homeomorphism, so that $X^u(\Per)/{\sim_0} \neq Y$ and $\tilde{g} \neq g$, see for example \cite[Page 14]{DGMW} where tiling spaces are discussed.

\begin{example}
Suppose $Y$ is a manifold and $g: Y \rightarrow Y$ is an expanding endomorphism as defined by Shub \cite{ShubExp}. In this case, $g$ is a covering map and hence a local homeomorphism. In \cite{DeeREU} it is shown that the map $g!$ is the transfer map associated to the cover, compare to Example \ref{localhomeoexample} above. In particular, $g!$ is a rational isomorphism in the case of $K$-theory and in the case of homology has even better properties, see \cite{DeeREU} for details.

The case when $Y$ is the circle and $g$ is the two-fold cover from the circle to itself is a special case of this situation. It is well-known or one can check directly that $g!$ is given by multiplication by 2 in degree zero and the identity in degree one. Hence,
\[
K_*(C^*(G^s(\Per)))\cong \left\{ \begin{array}{cc} \Z \left[ \frac{1}{2} \right] & *=0 \\ \Z & *=1 \end{array} \right.
\]
and
\[
K_*(C^*(G^s(\Per))\rtimes \Z)\cong \left\{ \begin{array}{cc} \Z & *=0 \\ \Z & *=1 \end{array} \right.
\]
The reader can see \cite{DeeREU} for other explicit computations of the transfer map in the case of flat manifolds.

The unstable and unstable Ruelle algebras associated to a Wieler solenoid are relevant in the study of the HK-conjecture. However, the stable and stable Ruelle algebras are not relevant for the HK-conjecture because the unit space of the relevant groupoids are not totally disconnected (except for the case of shifts of finite type). For example, in the case of a flat manifold $Y$, we have that the unit space is $\R^{{\rm dim}(Y)}$. 

\end{example}
\begin{example}
When the relevant Smale space is a two-sided shifts of finite type, the relevant Wieler pre-solenoid is the one-sided shifts of finite type, see \cite{Wie}. In this case, the map $g$ is the shift map and $g!$ is not a rational isomorphism even though $g$ is a local homeomorphism. This fact is in contrast with the previous example.
\end{example}

\begin{example}
We will discuss the case of the $p/q$-solenoid \cite{MR3628917} briefly. The details will be published elsewhere and were obtained at the same time as \cite{DeeREU}. As such, we give a short summary of the results. Let $S^1$ be the unit circle in the complex plane and $1<q<p$ be positive integers with gcd($p$, $q$)=1. The $p/q$-solenoid can be realized as an inverse limit where $Y$ is the $q$-solenoid. That is, 
\[
Y=S_q= \{ (z_0, z_1, z_2, \ldots) \mid z_i \in S^1 \hbox{ and }z_{i+1}^q=z_i \}
\]
The map is defined via
\[
g(z_0, z_1, z_2, \ldots)=(z_1^p, z_2^p, \ldots)
\]
where $(z_0, z_1, z_2, \ldots) \in Y=S_q$. 

Since $g$ is a local homeomorphism in this example, $K_*(C^*(G_0(\Per))\cong K^*(S_q)$. Furthermore, the $K$-theory of $S_q$ is known and given by
\[
K^*(S_q) \cong \left\{ \begin{array}{cc} \Z  & *=0 \\ \Z \left[ \frac{1}{q} \right] & *=1 \end{array} \right.
\]
Then one computes that $g!$ is multiplication by $p$ in degree zero and the identity in degree one. Hence, 
\[
K_*(C^*(G^s(\Per)))\cong \left\{ \begin{array}{cc} \Z \left[ \frac{1}{p} \right] & *=0 \\ \Z \left[ \frac{1}{q} \right] & *=1 \end{array} \right.
\]
and
\[
K_*(C^*(G^s(\Per))\rtimes \Z)\cong \left\{ \begin{array}{cc} \Z \left[ \frac{1}{q} \right] & *=0 \\ \Z \left[ \frac{1}{q} \right] \oplus \Z / (p-1)\Z & *=1 \end{array} \right.
\]
\end{example}

\begin{example}
When $(Y, g)$ is the aab/ab-solenoid the map $g$ is not a local homeomorphism. The $K$-theory of $C^*(G_0(\Per))$ and the induced map $\tilde{g}!$ were computed in \cite{DeeYas}. A summary is as follows. We have that 
\[
K_*(C^*(G_0(\Per))) \cong \left\{ \begin{array}{cc} \Z \oplus \Z  & *=0 \\ \Z& *=1 \end{array} \right.
\]
with $\tilde{g}!$ given by $\begin{bmatrix}
2 & 1\\
1 & 1
\end{bmatrix}$ in degree zero and the identity in degree one. Hence, 
\[
K_*(C^*(G^s(\Per))) \cong \left\{ \begin{array}{cc} \Z \oplus \Z  & *=0 \\ \Z& *=1 \end{array} \right.
\]
and 
\[
K_*(C^*(G^s(\Per))\rtimes \Z)\cong \left\{ \begin{array}{cc} \Z & *=0 \\ \Z & *=1 \end{array} \right.
\]
Other similar computations, both for one dimensional solenoids and more general constructions from tiling spaces, can be found in \cite{Gon1, Gon2, RenAP, ThoAMS, ThoSol, Yi}.
\end{example}

\end{document}